\documentclass{article}
\usepackage{amssymb,amscd,amsmath,amsthm,xcolor}
\usepackage{graphics}
\usepackage[dvips]{graphicx}
\usepackage{hyperref}
\usepackage[noabbrev,capitalise]{cleveref}
\usepackage{url}
\usepackage{authblk}
\usepackage{ textcomp }

\usepackage{pst-all}

\def\PP{\mathbb{P}}

\newcommand{\Psf}{\mathsf{P}}
\newcommand{\Qsf}{\mathsf{Q}}

\def\F{\mathcal{F}}

\def\C{\mathcal{C}}
\def\I{\mathcal{I}}

\def\L{\mathcal{L}}

\def\W{\mathcal{W}}

\def\P{\mathcal{P}}
\def\M{\mathcal{M}}

\def\Hc{\mathcal{H}}

\newcommand{\cs}{2^{\NN}}

\newcommand{\bstr}{2^{<\NN}}

\newcommand{\uh}{\upharpoonright}

\newcommand{\qvdash}{\operatorname{{?}{\vdash}}}
\newcommand{\nqvdash}{\operatorname{{?}{\nvdash}}}

\newcommand{\ISig}{\mathsf{I}\Sigma^0}

\newcommand{\BSig}{\mathsf{B}\Sigma^0}

\newcommand{\isig}{\mathsf{I}\Sigma}
\newcommand{\bsig}{\mathsf{B}\Sigma}

\newcommand{\RCA}[0]{\mathsf{RCA}}
\newcommand{\WKL}[0]{\mathsf{WKL}}

\newcommand{\CAC}[0]{\mathsf{CAC}}
\newcommand{\ADS}[0]{\mathsf{ADS}}

\newcommand{\EM}[0]{\mathsf{EM}}

\newcommand{\RT}[0]{\mathsf{RT}}
\newcommand{\SRT}[0]{\mathsf{SRT}}

\newcommand{\NN}[0]{\mathbb{N}}

\newcommand{\card}{\operatorname{card}}

\def\qt#1{``#1''}%

\usepackage{mathtools}

\def\N{\mathbb{N}}
\def\NN{\mathbb{N}}
\usepackage{helvet}

\def\PP{\mathbb{P}}

\def\F{\mathcal{F}}

\def\C{\mathcal{C}}
\def\I{\mathcal{I}}

\def\P{\mathcal{P}}
\def\M{\mathcal{M}}

\def\Hc{\mathcal{H}}
\def\PP{\mathbb{P}}

\newcommand{\R}{\mathcal{R}}
\usepackage{algorithm2e}
\RestyleAlgo{ruled}
\usepackage{babel}[french]
\usepackage{tikz}
\renewcommand{\L}[0]{\mathcal{L}}

\newcommand{\dom}{\operatorname{dom}}

\usepackage{xcolor} 

\newcommand{\COH}{\mathsf{COH}}

\newcommand{\DNCS}[1]{#1\mbox{-}\mathsf{DNC}}

\newcommand{\HEM}{\mathsf{HEM}}

\def\qt#1{``#1''}%

\newtheorem{theorem}{Theorem}
\numberwithin{theorem}{section}
\newtheorem{maintheorem}[theorem]{Main Theorem}
\newtheorem{lemma}[theorem]{Lemma}
\newtheorem{proposition}[theorem]{Proposition}
\theoremstyle{definition}
\newtheorem{definition}[theorem]{Definition}
\newtheorem{corollary}[theorem]{Corollary}
\newtheorem{example}[theorem]{Example}

\newtheorem{remark}[theorem]{Remark}
\newtheorem{statement}[theorem]{Statement}

\makeatletter
\newtheorem*{rep@theorem}{\rep@title}
\newcommand{\newreptheorem}[2]{%
\newenvironment{rep#1}[1]{%
 \def\rep@title{#2 \ref{##1}}%
 \begin{rep@theorem}}%
 {\end{rep@theorem}}}
\makeatother

\title{Ramsey-like theorems and immunities}

\author{Ahmed Mimouni \and Ludovic Patey}

\newreptheorem{maintheorem}{Main Theorem}
\begin{document}

\maketitle

\begin{abstract}
A Ramsey-like theorem is a statement of the form \qt{For every 2-coloring of $[\NN]^2$, there exists an infinite set~$H \subseteq \NN$ such that $[H]^2$ avoids some pattern}. We prove that none of these statements are computably trivial, by constructing a computable 2-coloring of~$[\NN]^2$ such that every infinite set avoiding any pattern computes a diagonally non-computable function relative to~$\emptyset'$. We also consider multiple notions of weaknesses based of variants of immunity, and characterize the Ramsey-like theorems which preserve these notions or not, based on the shape of the avoided pattern. This is part of a larger study of the reverse mathematics of Ramsey-like theorems.
\end{abstract}

\section{Introduction}

Among the theorems studied in reverse mathematics, Ramsey's theorem for pairs and two colors plays an important role, as it is historically the first statement proven to escape the structural phenomenon known as the \qt{Big Five}~\cite{simpson_2009,montalban2011open}. Ramsey's theorem for pairs and two colors ($\RT^2_2$) states that for every 2-coloring of $[\NN]^2$ --- the set of the unordered pairs over~$\NN$ ---, there exists an infinite set $H \subseteq \NN$ such that $[H]^2$ is monochromatic. The meta-mathematical study of Ramsey's theorem for pairs and its consequences motivated the discovery of many tools in computability theory, proof theory and combinatorics, among others~\cite{seetapun1995strength,cholak_jockusch_slaman_2001,liu2012rt22,monin2021srt22,patey2018proof}. 

In this article, we study a generalization of $\RT^2_2$ to a family of statements $\RT^2_2(p)$ stating that every coloring~$f : [\NN]^2 \to 2$ admits an infinite subset $H \subseteq \NN$ avoiding the pattern~$p$, where a pattern~$p : [\ell]^2 \to 2$ is a finite 2-coloring for some~$\ell \in \NN$, and avoiding the pattern~$p$ means that $p$ does not embed to $f \uh [H]^2$. These statements are referred to as \emph{Ramsey-like theorems}. We prove lower bounds on the statements $\RT^2_2(p)$ in a strong sense: there exists a computable coloring $f : [\NN]^2 \to 2$ such that every infinite set avoiding any pattern computes a diagonally non-computable function relative to~$\emptyset'$, that is, a function $g : \NN \to \NN$ such that $g(e) \neq \Phi_e^{\emptyset'}(e)$ for every~$e \in \NN$. We also prove that $\RT^2_2(p)$ does not admit probabilistic solutions, in that there exists a computable coloring $f : [\NN]^2 \to 2$ such that the measure of oracles computing an infinite set avoiding any pattern is~0. Beyond those uniform lower bounds, our main contributions is a characterization of which Ramsey-like theorems $\RT^2_2(p)$ preserve various notions of immunity, based on the shape of the pattern~$p$. More precisely, we consider preservation of $\omega$ hyperimmunities~\cite{patey2017iterative,downey2022relationships}, preservation of 2-dimensional hyperimmunity~\cite{liu2022reverse} and preservation of $\omega$ 2-dimensional hyperimmunities. The characterization for preservation of $\omega$ hyperimmunities is used in a follow-up article by Le Houérou and Patey~\cite{houerou2025ramseylike} to identify the Ramsey-like theorems equivalent to Ramsey's theorem for pairs over $\omega$-models.

\subsection{Reverse mathematics}

Our main motivation is reverse mathematics, but since we mainly consider separations over $\omega$-models, we shall place ourselves in the standard computability-theoretic realm, with no induction considerations.
\emph{Reverse mathematics} is a foundational program started by Harvey Friedman, whose goal is to find optimal axioms to prove ordinary theorems. It uses sub-systems of second-order arithmetic, with a base theory, $\RCA_0$, standing for Recursive Comprehension Axiom, which arguably captures \qt{computable mathematics}. See Simpson~\cite{simpson_2009} for a presentation of early reverse mathematics, with the Big Five phenomenon, and Dzhafarov and Mummert~\cite{dzhafarov2022reverse} for a more recent introduction.

Models of second-order arithmetic are of the form $\M = (M, S, +, \cdot, <)$, where $M$ denotes the sets of integers, $S \subseteq \P(M)$ is the collection of sets, $+$, $\cdot$ are binary operators and $<$ is a binary relation. The structure $(M, +, \cdot, <)$ is also called the \emph{first-order part} of~$\M$, and $S$ its \emph{second-order part}. When proving a non-implication over~$\RCA_0$, one prefers to witness the separation by a model as close to the standard interpretation as possible. We shall therefore mostly consider \emph{$\omega$-models}, that is, structures of the form $(\omega, S, +, \cdot, <)$, where $\omega$ is the set of standard integers, and $+, \cdot$ and $<$ have their usual interpretation. An $\omega$-model is therefore fully specified by its second-order part, and often identified with it.

Friedman characterized the second-order parts of $\omega$-models of~$\RCA_0$. A \emph{Turing ideal} is a non-empty collection of sets $\I$ that is downward-closed under the Turing reduction ($\forall X \in \I\ \forall Y \leq_T X\ Y \in \I$) and closed under the effective join ($\forall X, Y \in \I\ X \oplus Y \in \I$), where $X \oplus Y = \{ 2n : n \in X \} \cup \{2n+1 : n \in Y \}$. An $\omega$-model satisfies $\RCA_0$ if and only if its second-order part is a Turing ideal. From now on, we shall always assume that the $\omega$-models satisfy $\RCA_0$.

\subsection{Ramsey's theorem for pairs}\label[section]{sec:rt22}

As stated earlier, $\RT_2^2$ holds an important place in reverse mathematics, as it belongs to its own subsystem.
Its study was shaped by numerous seminal papers and long-standing open questions. Jockusch~\cite{jockusch1972ramsey} studied arithmetic bounds for Ramsey's theorem for pairs, by proving that every computable instance of $\RT^2_2$ admits a $\Pi^0_2$ solution, while there exists a computable instance with no $\Sigma^0_2$ solution. By \emph{instance}, we mean a coloring $f : [\NN]^2 \to 2$, while a \emph{solution} to~$f$ is an infinite homogeneous set. Seetapun~\cite{seetapun1995strength} then showed that $\RT^2_2$ does not imply the arithmetic comprehension scheme over~$\RCA_0$ by proving so-called \emph{cone avoidance} of~$\RT^2_2$, that is, for every non-computable set~$C$ and every computable instance of~$\RT^2_2$, there exists a solution~$H$ such that $C \not \leq_T H$. Cholak, Jockusch and Slaman~\cite{cholak_jockusch_slaman_2001} extensively studied $\RT^2_2$ both from a computability-theoretic and proof-theoretic viewpoint, and introduced the decomposition in terms of stability and cohesiveness. Then, Liu~\cite{liu2012rt22} proved that $\RT^2_2$ does not imply weak K\"onig's lemma over~$\RCA_0$ using PA avoidance. Nowadays, the reverse mathematics of Ramsey's theorem for pairs are relatively well understood, but few important open problems remain, such as the characterization of its first-order consequences~\cite{kolo2021search}.

In order to better understand the strength of Ramsey's theorem for pairs and its role in the Big Five phenomenon, multiple of its consequences were studied, among which the Erd\H{o}s-Moser theorem~\cite{erdos1964representation} and the Ascending Descending Sequen\-ce principle. A set $H$ is \emph{transitive} for a coloring $f : [\NN]^2 \to 2$ if for every $x < y < z \in H$ and every~$i < 2$, $f(x, y) = i$ and $f(y, z) = i$ implies $f(x, z) = i$. The Erd\H{o}s-Moser theorem ($\EM$) is a statement from graph theory stating that every coloring admits an infinite transitive set. The Ascending Descending Sequence principle ($\ADS$) says that every infinite linear order admits an infinite ascending or descending sequence. Bovykin and Weiermann (and Mont\'alban)~\cite{bovykin2017strength} showed that $\RT^2_2$ admits a natural decomposition in terms of $\EM$ and $\ADS$. Indeed, given an instance $f : [\NN]^2 \to 2$ of~$\RT^2_2$, see it as an instance of~$\EM$, to get an infinite transitive set~$G = \{ x_0 < x_1 < \dots \} \subseteq \NN$. Such a transitive set induces a linear order $\L = (\NN, <_\L)$ as follows: for every~$i <_\NN j$, let $i <_\L j$ iff $f(x_i, x_j) = 0$. For every infinite $\L$-ascending or $\L$-descending sequence $H$, the set $\{ x_i : i \in H \}$ is an infinite $f$-homogeneous set. Lerman, Solomon and Towsner~\cite{lerman2013separating} and Hirschfeldt and Shore~\cite{hirschfeldt2008strength} proved that this decomposition is non-trivial, in that neither $\EM$ nor $\ADS$ implies $\RT^2_2$ over~$\RCA_0$.

\subsection{Ramsey-like theorems}

Both the Erd\H{o}s-Moser theorem and the Ascending Descending Sequence principle can be seen as weakening of Ramsey's theorem for pairs, formulated in terms of avoidance of forbidden patterns. As mentioned, a \emph{pattern} is a 2-coloring of pairs over a finite initial segment of~$\NN$, i.e.\ a function $p : [\ell]^2 \to 2$ for some~$\ell \geq 1$. Then $\ell$ is called the \emph{length} of the pattern~$p$, written $|p|$. 
Let $p$ be a pattern of length~$\ell$, and $f : [\NN]^2 \to 2$ be a coloring. We say that a finite set $F = \{ x_0 < \dots < x_{\ell-1} \}$ \emph{$f$-realizes} $p$ if for every~$\sigma \in [\ell]^2$, $f(\{x_i : i \in \sigma\}) = p(\sigma)$. We say that~$H \subseteq \NN$ \emph{$f$-avoids} the pattern $p$ if no subset of~$H$ $f$-realizes~$p$.
The Erd\H{o}s-Moser theorem can then be rephrased as \qt{For every coloring $f : [\NN]^2 \to 2$, there exists an infinite set $f$-avoiding the two patterns of \Cref{fig:non-transitivity}.} The Ascending Descending Sequence principle can be seen as the dual statement \qt{For every coloring $f : [\NN]^2 \to 2$ avoiding the patterns of \Cref{fig:non-transitivity}, there exists an infinite homogeneous set.}

\begin{figure}[htbp]
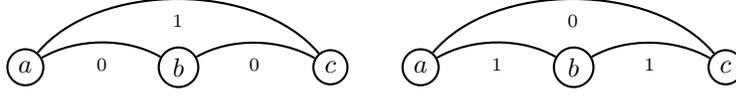

\vspace{0.7cm}
$$
\psmatrix[colsep=1.5cm,rowsep=1.5cm,mnode=circle]
a & b & c
\everypsbox{\scriptstyle} 
\ncarc[arcangle=30]{1,1}{1,2}_{0} 
\ncarc[arcangle=30]{1,2}{1,3}_{0}
\ncarc[arcangle=50]{1,1}{1,3}_{1} 
\endpsmatrix \hspace{20pt}
\psmatrix[colsep=1.5cm,rowsep=1.5cm,mnode=circle]
a & b & c
\everypsbox{\scriptstyle} 
\ncarc[arcangle=30]{1,1}{1,2}_{1} 
\ncarc[arcangle=30]{1,2}{1,3}_{1}
\ncarc[arcangle=50]{1,1}{1,3}_{0} 
\endpsmatrix
$$
\caption{The two graphs represent the patterns $p_0 : [3]^2 \to 2$ and $p_1 : [3]^2 \to 2$, respectively defined as follows: $p_0(0, 1) = p_0(1, 2) = 0$, $p_0(0,2) = 1$, $p_1(0, 1) = p_1(1, 2) = 1$ and $p_1(0, 2) = 0$. The domain $\{0,1,2\}$ corresponding to the set of vertices is renamed $\{a, b, c\}$ for readability.}\label{fig:non-transitivity}
\end{figure}

In this article, we only consider Ramsey-like theorems of the first kind, that is, the following family of statements:

\begin{statement}
For every pattern~$p$, $\RT^2_2(p)$ is the statement \qt{Every coloring $f : [\NN]^2 \to 2$ has an infinite set $f$-avoiding~$p$.}    
\end{statement}

First, note that Ramsey's theorem for pairs does not seem to be cast in this setting, as it states the existence of an infinite set avoiding either of the two constant patterns of length~2. We shall however see through \Cref{lem:avoiding-join} that for every pair of patterns $p, q$, there exists another pattern $p \uplus q$ such that the existence of an infinite set avoiding $p \uplus q$ is equivalent to the existence of an infinite set avoiding $p$ or $q$. Thus, every disjunctive pattern avoidance statement belongs to this family of Ramsey-like theorems.

Patey~\cite{patey_ramsey-like_2020} first introduced a more general family of Ramsey-like theorems where arbitrary many patterns can be avoided simultaneously, and characterized the patterns admitting strong cone avoidance. Ramsey-like theorems for patterns of size~3 and their dual statements were studied independently by multiple authors. See Cervelle, Gaudelier and Patey~\cite[Section 1.2]{cervelle2000reverse} for a survey.
In this article, our main contributions are of two kind. First, we prove uniform lower bounds to the family of Ramsey-like theorems: A function $g : \NN \to \NN$ is \emph{diagonally non-computable} relative to~$X$ (or $X$-DNC) if for every~$e \in \NN$, $g(e) \neq \Phi^X_e(e)$.

\begin{maintheorem}\label{mainthm:dnczp}
     There exists a computable coloring $f : [\NN]^2 \to 2$ such that every infinite set~$H \subseteq \NN$ avoiding any pattern for $f$ computes a $\emptyset'$-DNC function.
\end{maintheorem}

In particular, such a coloring admits no $\Sigma^0_2$ infinite set avoiding any pattern, as there is no $\Sigma^0_2$ $\emptyset'$-DNC function. Note that for every set~$X$, there exists a probabilistic algorithm to compute an $X$-DNC function: given $n \in \NN$, let $g(n)$ be a value picked at random within $[0, 2^{n+2})$. Then the probability of failure is bounded by $\sum_n 2^{-n-2} = 0.5$. Our next lower bound shows that there is no probabilist algorithm avoiding a pattern in general, as there exists a coloring such that the class of the sets avoiding any pattern is of measure $0$ :

\begin{maintheorem}\label{mainthm:measure0}
  There is a computable coloring $f : [\NN]^2 \to 2$ such that the measure of oracles computing an infinite set avoiding any pattern for~$f$ is~0.
\end{maintheorem}

Our second main contribution is the characterization of which Ramsey-like theorems admit some notion of preservation, in a sense that we define now.

\subsection{Preservation of immunities}\label[section]{sec:preservation-immunities}

Many statements studied in reverse mathematics are $\Pi^1_2$ sentences which can be seen as mathematical problems, in terms of instances and solutions. 
A proof that a problem~$\Psf$ does not imply another problem~$\Qsf$ over $\RCA_0$ usually consists of creating an $\omega$-model by iteratively adding solutions to instances of~$\Psf$, while avoiding adding solutions to a fixed instance of $\Qsf$. For this, it is useful to identify a property which is preserved by $\Psf$, but not by $\Qsf$.

\begin{definition}
A \emph{weakness property} is a non-empty class of sets which is downward-closed under the Turing reduction.
\end{definition}

There exist many natural weakness properties, such as being $\Delta^0_n$ for some~$n \geq 1$, being of low degree, or even not computing the halting set. Any Turing ideal is a weakness property, but most natural weakness properties are not closed under effective join. For instance, the join of two low sets can compute the halting set.

\begin{definition}\label[definition]{def:preservation}
A problem~$\Psf$ \emph{preserves} a weakness property~$\W \subseteq \cs$ if for every~$Z \in \W$ and every $Z$-computable $\Psf$-instance~$X$, there is a $\Psf$-solution~$Y$ to~$X$ such that $Y \oplus Z \in \W$.
\end{definition}

Intuitively, a problem $\Psf$ preserves a weakness property if every weak instance admits a weak solution, in the appropriate relative form. By a classical argument (see~\cite[Section 3.4]{patey2016reverse}), if~$\Psf$ preserves $\W$, then for every~$Z \in \W$, there exists an $\omega$-model of~$\RCA_0 + \Psf$ containing~$Z$. It follows that if $\Psf$ preserves~$\W$ but $\Qsf$ does not, then there is an $\omega$-model of~$\RCA_0 + \Psf$ which is not a model of~$\Qsf$.

The weakness properties considered to separate theorems from Ramsey theory are usually formulated in terms of variants of immunity. An infinite set~$A \subseteq \NN$ is \emph{immune} if it does not contain any infinite computable subset. If an instance of a Ramsey-like theorem $\RT^2_2(p)$ admits no computable solution, then any solution is immune, as the class of solutions is closed under infinite subsets. We shall consider two variants of immunity, inducing many preservation properties.

\begin{definition}
An \emph{array} is a collection of non-empty finite sets $\vec{F} = \langle F_n : n \in \NN \rangle$ such that $\min F_n > n$ for every~$n \in \NN$. An array is c.e. if the function which maps~$n$ to a canonical code of~$F_n$ is computable. An infinite set $A \subseteq \NN$ is \emph{$X$-hyperimmune} if for every $X$-c.e.\ array $\vec{F}$, there is some~$n \in \NN$ such that $X \cap F_n = \emptyset$.
\end{definition}

Hyperimmunity is a strong form of immunity, as one cannot even approximate an infinite subset by blocks. There exists a characterization of the Turing degrees of hyperimmune sets in terms of function domination.
A function $f : \NN \to \NN$ \emph{dominates} $g : \NN \to \NN$ if $f(n) \geq g(n)$ for every~$n \in \NN$.
A function $f$ is \emph{$X$-hyperimmune} if it is not dominated by any $X$-computable function. The degrees computing an $X$-hyperimmune set and those computing an $X$-hyperimmune function coincide.

\begin{definition}Fix $k \in \NN \cup \{\NN\}$.
A problem~$\Psf$ \emph{preserves $k$ hyperimmunities} if for every set~$Z$,
every family of $Z$-hyperimmune functions $\langle f_s : s < k \rangle$, and every $Z$-computable $\Psf$-instance~$X$, there exists a $\Psf$-solution~$Y$ to~$X$ such that each~$f_s$ is $Y \oplus Z$-hyperimmune.
\end{definition}

Preservation of $k$ hyperimmunities is a scheme of preservations in the sense of \Cref{def:preservation}. Thus, if a problem~$\Psf$ preserves~$k$ hyperimmunities for some~$k \in \NN \cup \{\NN\}$ but $\Qsf$ does not, then $\Psf$ does not imply $\Qsf$ over $\RCA_0$. These notions were formally introduced by Patey~\cite{patey2017iterative}, although the ideas were already present in the combinatorics of Lerman, Solomon and Towsner~\cite{lerman2013separating}. They were later systematically studied by Downey et al.~\cite{downey2022relationships}. In particular, preservation of 2 or more hyperimmunities is the most convenient tool to separate a statement from Ramsey's theorem for pairs and two colors. Patey~\cite{patey2017iterative} proved that $\EM$ preserves $\omega$ hyperimmunities, while $\RT^2_2$ (and $\ADS$) does not even preserve 2 hyperimmunities. Our last main theorem is a characterization of which Ramsey-like theorems preserve $\omega$ hyperim\-munities. The notions of irreducible and divergent pattern will be defined in \Cref{sect:general-properties}.

\begin{maintheorem}\label{omega-hyp-divergent-irreducible}
Let $p$ be a pattern. $\RT^2_2(p)$ preserves $\omega$ hyperimmunities if and only if $p$ contains a sub-pattern which is simultaneously divergent and irreducible.
\end{maintheorem}

The relevance of this theorem is justified by a follow-up paper by Le Houérou and Patey~\cite{houerou2025ramseylike} in which they prove that $\RT^2_2(p)$ implies $\RT^2_2$ over $\omega$-models if and only if $\RT^2_2(p)$ does not preserve $\omega$ hyperimmunities.

We also prove similar characterization theorems for other notions of preserva\-tions, namely, preservation of $k$ 2-dimensional hyperimmunities, defined by Liu and Patey~\cite{liu2022reverse} to separate theorems from the Erd\H{o}s-Moser theorem. Although this notion might seem much more ad-hoc, it is arguably the natural property induced by the combinatorics of~$\EM$. This will in particular used in \Cref{sec:hem} to separate from $\EM$ an asymmetric version of the Erd\H{o}s-Moser theorem ($\HEM$), in which the solution needs to be transitive for only one color. This statement~$\HEM$, together with the Chain AntiChain principle ($\CAC$), forms another decomposition of~$\RT^2_2$. This decomposition is of particular interest, as $\CAC$ is the strongest consequence of~$\RT^2_2$ for which the first-order part is known (see Chong, Slaman and Yang~\cite{chong2021pi11}).

\subsection{Organization of the paper}

We start by proving in \Cref{sec:avoiding-any} two uniform lower bounds on the strength of~$\RT^2_2(p)$ for any pattern~$p$. Then, we study basic properties of Ramsey-like theorems in \Cref{sect:general-properties}, and define in particular the join operator~$\uplus$. Then, \Cref{sect:omega-hyp,sect:2dim-hyp,sect:omega-2dim-hyp} are devoted to the characterization of which Ramsey-like theorems preserve $\omega$ hyperimmunities, one 2-dimensional hyperimmunity and $\omega$ 2-dimen\-sional hyperimmunities, respectively. Last, \Cref{sec:hem} studies an asymmetric version of the Erd\H{o}s-Moser theorem, namely, $\HEM$.

\section{Avoiding any pattern}\label[section]{sec:avoiding-any}

We now prove two lower bounds to the strength of Ramsey-like theorems, in terms of diagonally non-computable functions and probabilistic algorithms. It follows that none of the Ramsey-like theorems are computably trivial, that is, for every pattern~$p$, there exists a computable coloring $f : [\NN]^2 \to 2$ with no computable infinite set avoiding~$p$. Interestingly, both lower bounds are uniform, in that the constructed coloring does not depend on the choice of~$p$.

\begin{definition}
A function $f : \NN \to \NN$ is \emph{diagonally non-$X$-computable} ($X$-DNC) if for every~$e \in \NN$, $f(e) \neq \Phi^X_e(e)$.
\end{definition}

Diagonally non-computable degrees play an important role in computability theory. They admit many characterizations, in terms of effective immunity (\cite{jockusch1989recursively}), fixpoint-free functions (\cite{jockusch1989recursively}), infinite subsets of random sequences (\cite{kjos2009infinite,greenberg2009lowness}, see \cite[Theorem 8.10.2]{downey2010algorithmic}) or Kolmogorov complexity (\cite{kjos2011kolmogorov}), among others. We shall actually use here the characterization in terms of fixpoint-free functions.

\begin{definition}
A function $f : \NN \to \NN$ is \emph{$X$-fixpoint free} ($X$-FPF) if for every $e \in \NN$, $W^X_{f(e)} \neq W^X_e$.
\end{definition}

By Jockusch~\cite{jockusch1989recursively}, the degrees computing an $X$-DNC function and those computing an $X$-FPF function coincide. 
Given a pattern~$p$ of size at least~2, we let $p^-$ be the restriction of~$p$ to the domain~$[|p|-1]^2$. We are now ready to prove our first main theorem.


\begin{repmaintheorem}{mainthm:dnczp}
    There exists a computable coloring $f : [\NN]^2 \to 2$ such that every infinite set~$H \subseteq \NN$ avoiding any pattern for $f$ computes a $\emptyset'$-DNC function.
\end{repmaintheorem}
\begin{proof}
Let us build the function $f : [\NN]^2 \to 2$ using a no-injury priority construc\-tion in the style of Jockusch~\cite[Theorem 3.1]{jockusch1972ramsey}.
Each requirement will be of the form $\R_{p,e  }$, where~$p$ is a pattern and $e$ a Turing index. The requirements are ordered using the Cantor pairing function $\langle p, e\rangle$, the least pair being the requirement of higher priority. At each stage~$s$, the construction will define the value of $f$ for each pair $\{x, s\}$ with $x < s$. Within a stage~$s$, each requirement $\R_{p,e  }$ will put a restrain on at most $|p|-1$ many elements $x < s$. Let $h : \NN \to \NN$ be the following computable function:
$$h(\langle p,e  \rangle) = \sum_{\langle q, i\rangle < \langle p,e \rangle} |q^-|
$$
Given a requirement $\R_{p,e  }$ acting at a stage~$s$, $h(\langle p,e \rangle)$ represents the maximum number of elements~$x < s$ restrained by a requirement of higher priority.
The function $f$ will satisfy the following requirements, for every pattern $p$ of size at least~2, and every code $e$:
   \begin{quote}
        $\R_{p,e  }$: Either $p^{-}$ appears at most $h(\langle p,e \rangle)$ many times in $W^{\emptyset'}_e$ with pairwise disjoint blocks of elements or there is no infinite set $H \supseteq W^{\emptyset'}_e$ avoiding the pattern $p$.
    \end{quote}

    For each~$e$ and $s$, let $W_{e}^{\emptyset'}[s] = \{ x < s : \Phi^{\emptyset'[s]}_e(x)[s]\downarrow \}$. Note that $W_{e}^{\emptyset'}[s] \leq s$. The \emph{$(e,s)$-age} of an element~$x$ is the biggest $t \leq s$ such that $x \in W_{e}^{\emptyset'}[r]$ for every~$r \in \{s-t, \dots, s\}$. Given a pattern~$p$, a Turing index~$e$ and an approximation stage~$s$, let $F_{p, e,s}$ be the set of $h(\langle p,e \rangle)+1$ many pairwise disjoint $(e,s)$-oldest blocks of elements realizing $p^{-}$, if they exist. If not, then $F_{p, e,s}$ is not defined. 
    In particular, $\{ F_{p, e,s} \}_{s \in \NN}$ is such that if $p^{-}$ appears at least $h(\langle p,e \rangle)+1$ many times in $W_e^{\emptyset'}$ with pairwise disjoint blocks of elements, then $\lim_s F_{p, e, s}$ exists and contains such witnesses.

    \smallskip

    \textbf{Construction}.
    We construct $f$ by stages. At the beginning of each stage, all the restraints are released. At stage $0$, $f_0$ is the empty function. At stage~$s > 0$, suppose $f_{s-1}$ is defined over $[\{0,\dots, s-1\}]^2$. We define $f_{s}(x, s)$ for every $x < s$ as follows:

    For each $\langle p,e \rangle < s$, if $F_{p, e,s}$ is defined, then pick a block $E_{p, e } \in F_{p, e,s}$ which does not contain any element restrained by a requirement of higher priority. Then, put a restraint on all the elements of $E_{p, e }$. 
    
    Note that $\max E_{p,e  } \leq s$. Also note that if $p^{-}$ appears at least $h(\langle p,e \rangle)+1$ many times in $W^{\emptyset'}_e$ with pairwise disjoint blocks of elements, then, as mentioned above, $F_{p, e, s}$ is defined for cofinitely many stages~$s$, and by a cardinality argument, such a block $E_{p, e } \in F_{p, e,s}$ exists, as there are at most $h(\langle p, e\rangle)$ elements restrained by requirements of higher priority. When~$F_{p, e, s}$ stabilizes, the elements restrained by higher priority arguments might still change infinitely often if, for instance, $W_i^{\emptyset'} = \emptyset$ for some~$i < e$. It follows that even after the stabilization stage, the choice of $E_{p, e}$ within $F_{p, e,s}$ might change infinitely often.

    Let $a_i$ be the $i$-th element of $E_{p,e  }$ in the natural order over the integers, for $i \geq 0$. Let $f(a_i,s) = p(i,|p|-1)$. Since $E_{p,e  }$ realizes $p^-$, $E \cup \{s\}$ realizes $p$.
    Then, assign any color to the unassigned pairs to complete $f_{s}$ over $[\{0,\dots, s\}]^2$. This completes the construction of~$f$.
    \smallskip

\textbf{Verification}. We claim that $f$ satisfies all the requirements. Indeed, fix some~$\langle p,e \rangle$ such that $p^{-}$ appears at least $h(\langle p,e \rangle)+1$ many times in $W^{\emptyset'}_e$ with pairwise disjoint sets of elements. Then, there is some~$t$ such that for every~$s \geq t$, $F_{p, e,s}$ contains such witnesses. Let $H$ be an infinite superset of~$W_e^{\emptyset'}$. Pick $s \in H$ larger than~$t$ and let $E_{p, e } \in F_{p, e,s}$ be the set chosen at stage~$s$. Then by construction, $E_{p, e } \cup \{s\}$ is a subset of~$H$ realizing $p$, so $H$ does not avoid the pattern~$p$.
    \smallskip

    We now claim that every infinite set avoiding any pattern for $f$ computes a $\emptyset'$-DNC function. We proceed by induction over an enumeration of the avoided pattern $p$, monotonous in the size of the patterns. Let $H$ be an infinite set avoiding $p$. 
    For the base case, $H$ cannot avoid the unique pattern~$p$ with only one node, so the property vacuously holds. Let $p$ be a pattern with at least two nodes. We consider two cases:
    \begin{itemize}
        \item \emph{Case 1:} $H$ contains arbitrarily many pairwise disjoint subsets realizing $|p^-|$. Then, for all $e \in \NN$, let $G_e$ be a union of $h(\langle p,e \rangle)+1$ such subsets. Since $G_e$ can be extended into $H$ an infinite set avoiding $p$, and contains $h(\langle p,e \rangle)+1$ pairwise disjoints subsets realizing $p^-$,  it cannot be equal to $W^{\emptyset'}_e$ by $\R_{p,e   }$. Consider the function $g : \NN \to \NN$ such that for every~$e \in \NN$, $W^{\emptyset'}_{g(e)} = G_e$. Such a function exists since $G_e$ is finite, hence $\emptyset'$-c.e. Moreover, the function $g$ is $H$-computable. This yields that for all $e$, $W^{\emptyset'}_{g(e)} \neq W^{\emptyset'}_e$. The function $g$ is an $H$-computable fixpoint-free function relative to~$\emptyset'$, proving $H$ computes a DNC function relative to~$\emptyset'$ (see Jockusch~\cite{jockusch1989recursively}). 
        \item \emph{Case 2:} Case 1 does not hold, i.e. finitely many disjoints subsets of $H$ realize $p^-$. Then let $t$ be such that no subset of $H \setminus \{0, \dots, t\}$ realizes $p^-$, and let $\hat H$ be $H \setminus \{ 0, \dots, t \}$. Then, no subset of $\hat H$ realizes $p^-$, hence, by induction hypothesis $\hat H$ computes a $\emptyset'$-DNC function, and since $H$ computes $\hat H$, $H$ also computes a $\emptyset'$-DNC function. 
    \end{itemize}

\end{proof}

From a reverse mathematical viewpoint, the previous construction uses $\BSig_2$ to obtain a robust formalization of computation relative to the jump, and uses once $\ISig_2$ over the size of the avoided pattern in the verification. Thus, if one considers a fixed pattern of standard size, we obtain the following proposition. Here, $\DNCS{n}$ is the statement \qt{For every set $X$, there is an $X^{(n-1)}$-DNC function}.

\begin{proposition}\label[proposition]{prop:rca-bsig2-rt22p-2dnc}
$\RCA_0 + \BSig_2 \vdash \RT^2_2(p) \to \DNCS{2}$.
\end{proposition}

Note that the previous lower bound does not rule out the existence of a probabilistic algorithm to find solutions to computable instances of $\RT^2_2(p)$, as there exists a probabilistic algorithm to compute a $\emptyset'$-DNC function (see \Cref{sec:rt22}). More precisely, for every set~$X$, the measure of oracles computing an $X$-DNC function is~1. We now prove that this lower bound is not tight, in that there exists a computable coloring such that no probabilistic algorithm can compute an infinite set avoiding any pattern.

\begin{theorem}\label[theorem]{thm:measure-avoid-pattern}
There is a computable coloring $f : [\NN]^2 \to 2$ such that the measure of oracles computing an infinite set avoiding any pattern for~$f$ is~0.
\end{theorem}
\begin{proof}
The coloring $f : [\NN]^2 \to 2$ is built using a finite-injury priority construction, satisfying the following requirements for every pattern~$p$ and every Turing index~$e$:
\begin{quote}
    $\R_{p,e}$: $\mu(\{ X \in \cs : W_e^X \mbox{ is finite or } f\mbox{-realizes } p \}) \geq \frac{1}{2|p|}$.
\end{quote}
We first claim that if all $\R$-requirements are satisfied, then $f$ satisfies the statement of the theorem, using the contrapositive.
Suppose that the measure of oracles computing an infinite set avoiding any pattern for~$f$ is positive.
Then, since there are countably many patterns and countably many functionals, there is some pattern~$p$ and some Turing index~$e$ such that the measure of oracles~$X$ such that $\Phi_e^X$ is infinite and $f$-avoids~$p$ is positive. Indeed, a countable union of classes of measure~0 is again of measure~0. By the Lebesgue density theorem, there is some string $\sigma \in \bstr$ such that  the measure of oracles $X$ such that $\Phi_e^{\sigma \cdot X}$ is infinite and $f$-avoids~$p$ is more than $1-\frac{1}{2|p|}$.
Let~$a$ be a Turing index such that $\Phi_a^X = \Phi_e^{\sigma \cdot X}$. Then the requirement $\R_{p, a}$ is not satisfied. 
\smallskip

The strategies are given a priority order based on Cantor's pairing function $\langle p, e\rangle$, the smaller value being of higher priority. Each requirement $\R_{p, e}$ will be given a movable marker~$m_{p_e}$, starting at~0, such that if $\R_{q, i}$ is of greater priority than $\R_{p, e}$, then $m_{q, i} \leq m_{p, e}$. Moreover, at any stage~$s$, $m_{p,e}$ is greater than any value restrained by $\R_{p, e}$ or any requirement of higher priority.
\smallskip

\textbf{State of $\R_{p,e}$.}
Each strategy $\R_{p, e}$ will be given a \emph{state}, which is a finite sequence $\langle F_0, \dots, F_{t-1}\rangle$ of non-empty finite sets, with $t \leq |p|$ and such that $\max F_i < \min F_{i+1}$. Such sequence will satisfy the following properties:
\begin{enumerate}
    \item[(P1)] For every~$x_0 \in F_0, \dots, x_{t-1} \in F_{t-1}$, $\{x_0, \dots, x_{t-1}\}$ $f$-realizes $p \uh_t$
    \item[(P2)] $\mu(\{ X : W_e^X \cap F_i \neq \emptyset \}) > 1-\frac{1}{2|p|}$
\end{enumerate}
Over time, new sets will be stacked to this list, which will be reset only if a strategy of higher priority injures it. Initially, each strategy is given the empty sequence as state.

\smallskip

\textbf{Strategy for $\R_{p,e}$.}
A requirement $\R_{p, e}$ \emph{requires attention at stage~$s$} if its state has length less than~$|p|$ and there is a finite set $D \subseteq 2^{\leq s}$ such that $\sum_{\sigma \in D} 2^{-|\sigma|} > 1-\frac{1}{2|p|}$ and for every~$\sigma \in D$, $W_e^\sigma[s] \cap [m_{p,e}, s] \neq \emptyset$. In other words, $\R_{p,e}$ requires attention at stage~$s$ if the measure of oracles $X$ such that $W_e^X[s]$ outputs an element in $[m_{p,e}, s]$ is greater than $1-\frac{1}{2|p|}$.

If $\R_{p,e}$ receives attention at stage~$s$ and is in state~$\langle F_0, \dots, F_{t-1} \rangle$ (by convention, if the state is the empty sequence, $t = 0$), then, letting $F_t = [m_{p,e}, s]$, its new state is $\langle F_0, \dots, F_t \rangle$. The marker $m_{p,e}$ is moved to $s+1$, and the markers of all strategies of lower priorities is moved further, accordingly. All the strategies of lower priorities are injured, and their state is reset to the empty sequence. If $t < |p|-1$, then for every~$i \leq t$, all the elements of $F_i$ commit to have limit $p(i, t+1)$.
\smallskip

\textbf{Construction.}
The global construction goes by stages, as follows. Initially, $f$ is nowhere-defined.
At stage~$s$, suppose $f$ is defined on $[\{0, s-1\}]^2$. If some strategy requires attention at stage~$s$, letting $\R_{p, e}$ be the strategy of highest priority among these, give it attention and act accordingly. 
In any case, for every~$x < s$, if $x$ is committed to have some limit~$c < 2$, then set $f(x, s) = c$.
Otherwise, set $f(x, s) = 0$. Go to the next stage.
\smallskip

\textbf{Verification.}
One easily sees by induction on the strategies that every strategy acts finitely often, and therefore each strategy is finitely injured by a strategy of higher priority. It follows that each requirement has a limit state and that $m_{p,e}$ reaches a limit value.

We claim by induction over~$t$ that at any stage~$s$, if a requirement $\R_{p,e}$ has state $\langle F_0, \dots, F_{t-1}\rangle$, then the sequence satisfies (P1). This is trivially the case for~$t \leq 1$. Suppose that $t > 1$, and let $s_0 < s$ be a prior stage at which $\R_{p,e}$ was first given the state $\langle F_0, \dots, F_{t-2}\rangle$. In particular, $\R_{p,e}$ was not injured between stage~$s_0$ and $s$, otherwise its marker~$m_{p,e}$ would have been moved to a value greater than~$s_0$, and it would never get the stage $\langle F_0 \rangle$ again. At stage~$s_0$, for every~$i \leq t-2$, $\R_{p,e}$ committed all the elements of~$F_i$ to have limit $p(i, t-1)$ and moved the marker $m_{p,e}$ to $s_0+1$. In particular, $\min F_{t_0} > s_0$. Since $\R_{p,e}$ was not injured between~$s_0$ and $s$ and $\min F_{t-1} > s_0$, then for every~$i \leq t-2$, $x \in F_i$ and $y \in F_{t-1}$, $f(x, y) = p(i,t-1)$. Thus, together with the induction hypothesis, (P1) holds for the sequence $\langle F_0, \dots, F_{t-1}\rangle$.

We claim that each requirement $\R_{p,e}$ is satisfied. Let $s$ be stage after which $m_{p, e}$ reaches its limit value. In particular, the state of~$\R_{p, e}$ also reached its limit, and none of the strategies of higher priority require attention after~$s$. Suppose first that the limit state of~$\R_{p, e}$ has length less than~$|p|$. This means that $\R_{p, e}$ does not require attention attention after stage~$s$, so the measure of oracles $X$ such that $W_e^X$ outputs an element greater than or equal to~$m_{p,e}$ is at most $1-\frac{1}{2|p|}$. Thus, $\mu(\{ X \in \cs : W_e^X \mbox{ is finite } \}) \geq \frac{1}{2|p|}$, and the requirement is therefore satisfied. Suppose now that the limit state of~$\R_{p,e}$ has length~$|p|$. By (P2), for each~$i < |p|$, $\mu(\{ X : W_e^X \cap F_i \neq \emptyset \}) > 1-\frac{1}{2|p|}$. It follows that 
$$
\mu(\{ X : (\forall i < |p|) W_e^X \cap F_i \neq \emptyset \}) > 1-\frac{|p|}{2|p|} = 1/2
$$
Thus, by (P1), $\mu(\{ X \in \cs : W_e^X f\mbox{-realizes } p \}) \geq \frac{1}{2}$, so $\R_{p,e}$ is again satisfied. This completes the proof of \Cref{thm:measure-avoid-pattern}.
\end{proof}

\section{General properties}\label[section]{sect:general-properties}

This section introduces some basic operators and definitions relative to patterns. This lays the groundwork necessary for the following sections.
Let us first remark that the computational strength of a pattern lies in its variation of color, and not in the colors themselves. In the case of 2-colorings, every pattern~$p$ has a \emph{dual} pattern~$\overline{p}$, obtained by flipping the color of every pair :  $\forall x,y < |p|,\ \overline{p}(x,y) = 1 - p(x,y)$. 

\begin{lemma}
For every pattern~$p$, $\RCA_0 \vdash \RT^2_2(p) \leftrightarrow \RT^2_2(\bar p)$. 
\end{lemma}
\begin{proof}
We prove $\RT^2_2(p) \to \RT^2_2(\bar p)$ as both statements play a symmetric role.
Let $f : [\NN]^2 \to 2$ be an instance of $\RT^2_2(\bar p)$.
Let $g : [\NN]^2 \to 2$ be defined by $g(x, y) = 1-f(x, y)$. Every infinite set $g$-avoiding $p$ also $f$-avoids $\bar p$.
\end{proof}


Another relation of importance between patterns is the subpattern relation: a pattern $q$ is a \emph{subpattern} of $p$ if there exists an injective function $g : |q| \to |p|$ such that for all $x,y < |q|,\ q(x,y) = p(g(x),g(y))$. 

\begin{lemma}\label[lemma]{lem:sub-pattern-implication}
    Let $p$ and $q$ be two patterns. If $q$ is a sub-pattern of $p$, then  $\RT^2_2(q)$ implies  $\RT^2_2(p)$.
\end{lemma}
\begin{proof}
Let $f : [\NN]^2 \to 2$ be a coloring and $H$ be an infinite set $f$-avoiding $q$.
Since $q$ is a sub-pattern of $p$, any set $f$-realizing $p$ contains a subset $f$-realizing~$q$,
so $H$ $f$-avoids~$p$.
\end{proof}


Remember that Ramsey-like theorems are non-disjunctive statements. As such, the formalism seems to capture a very restrained family of statements from Ramsey theory, as many of them are naturally stated in a disjunctive form. For instance, Ramsey's theorem for pairs states the existence of an infinite set avoiding any of the two constant patterns of size~2. Thankfully, the following join operator enables to cast disjunctive statements into the framework.


\begin{definition}
Let $p$ and $q$ be two patterns. Then \emph{join} $p \uplus q$ is the pattern of size $|p|+|q|-1$
defined for every~$x < y < |p|+|q|-1$ by
$$
(p \uplus q)(x, y) = \left\{ \begin{array}{ll}
    p(x, y) & \mbox{ if } y < |p|\\
    q(x-|p|+1, y-|p|+1) & \mbox{ if } x \geq |p|-1\\
    p(x, |p|-1) & \mbox{ if } x < |p|-1 \mbox{ and } y > |p|-1
\end{array}\right.
$$
\end{definition}

In other words, the join $p \uplus q$ is obtained by merging the right-most node of the graph of~$p$ with the left-most node of the graph of~$q$, and letting every arrow between some node~$x$ of~$p$ and some node $y$ of~$q$ have the value $p(x, |p|-1)$.

\begin{example}
The following pattern of length~4
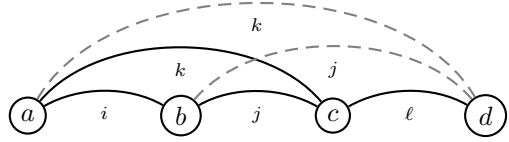
\vspace{1.2cm}
$$
\psmatrix[colsep=1.5cm,rowsep=1.5cm,mnode=circle]
a & b & c & d
\everypsbox{\scriptstyle} 
\ncarc[arcangle=30]{1,1}{1,2}_{i} 
\ncarc[arcangle=30]{1,2}{1,3}_{j}
\ncarc[arcangle=30]{1,3}{1,4}_{\ell}
\ncarc[arcangle=50]{1,1}{1,3}_{k} 
\ncarc[arcangle=50, linestyle=dashed, linecolor=gray]{1,2}{1,4}_{j} 
\ncarc[arcangle=60, linestyle=dashed, linecolor=gray]{1,1}{1,4}_{k} 
\endpsmatrix
$$
is the join of the following two patterns
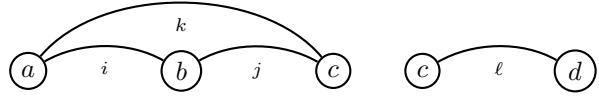
\vspace{0.7cm}
$$
\psmatrix[colsep=1.5cm,rowsep=1.5cm,mnode=circle]
a & b & c
\everypsbox{\scriptstyle} 
\ncarc[arcangle=30]{1,1}{1,2}_{i} 
\ncarc[arcangle=30]{1,2}{1,3}_{j}
\ncarc[arcangle=50]{1,1}{1,3}_{k} 
\endpsmatrix
\hspace{20pt}
\psmatrix[colsep=1.5cm,rowsep=1.5cm,mnode=circle]
c & d
\everypsbox{\scriptstyle} 
\ncarc[arcangle=30]{1,1}{1,2}_{\ell}
\endpsmatrix
$$ 
\end{example}

The following lemma, of central importance, states that the strength of avoiding a pattern of the form~$p \uplus q$ can be understood in terms of avoidance of~$p$ and~$q$.

\begin{lemma}\label[lemma]{lem:avoiding-join}
Let $p$ and $q$ be two patterns.
Let $f : [\NN]^2 \to 2$ be a 2-coloring and $H$ be an infinite set and avoiding $p \uplus q$.
Then there is an infinite $f \oplus H$-computable subset~$Y \subseteq H$ which avoids either $p$, or $q$.
\end{lemma}
\begin{proof}
Suppose first that for every finite subset $F \subseteq H$ which avoids pattern $p$ there is some $z \in H \setminus F$ such that $F \cup \{ z \}$ avoids $p$. Then, by a greedy construction, one can $f \oplus H$-compute an infinite subset of $H$ which avoids $p$.

Contrarily, suppose that there exists a finite set $F \subseteq H$ which avoids $p$ such that for every $z \in H$, the set $F \cup \{z \}$ does not avoid $p$. By multiple applications of the pigeonhole principle, there exists an infinite $f \oplus H$-computable subset $Z \subseteq H$ such that $\min Z > \max F$ and such that for every $x \in F$, the function $f(x, \cdot)$ is constant over $Z$.

Now, suppose $q$ occurs in $Z$, as a finite set $Q \subseteq Z$. Then, by construction, there exists a finite set $P \subseteq F$ such that $P \cup \{ \min Q \}$ $f$-realizes $p$, and such that for all $x \in P$ and $y \in Q$, $f(x,y) = f(x, \min Q)$. This would make $p \uplus q$ occur in $F \cup Z \subseteq H$, contradicting assumption. This yields that $q$ does not occur in~$Z$.
\end{proof}

Note that given a join pattern $p \uplus q$, some colorings will yield sets avoiding $p$, and others $q$. As such, the previous lemma does not prove that $\RT_2^2(p \uplus q)$ implies $\RT_2^2(p) \lor \RT_2^2(q)$, but rather that $\RT^2(p \uplus q)$ implies the disjunctive statement \qt{For every 2-coloring of pairs $f$, there exists an infinite set $f$-avoiding either $p$ or $q$.}

Every pattern is the join of itself and the trivial pattern of length~1. By \Cref{lem:avoiding-join}, the avoidance of a pattern of the form $p \uplus q$ is reduced to the avoidance of the patterns~$p$ and~$q$. As a consequence, the patterns which cannot be expressed as a join of two smaller patterns should receive a particular attention. This motivates the following definition:

\begin{definition}
A pattern is \emph{reducible} if it is of the form $p \uplus q$, with $|p|, |q| \geq 2$.
Otherwise, it is \emph{irreducible}.
\end{definition}

The following technical lemma is essentially an explicit formulation of the notion of irreducibility by unfolding the definition. It will be useful in the later sections. 
\begin{lemma}\label[lemma]{lem:irreducible-is-extensible}
A pattern~$p : [\ell]^2 \to 2$ is irreducible if and only if for every 2-partition $F \sqcup G = \ell$ such that $F \neq \emptyset$, $\card G \geq 2$, and $F < G$, there is some~$x \in F$ and $y, z \in G$ such that $p(x, y) \neq p(x, z)$.
\end{lemma}
\begin{proof}
    Suppose first $p = p_0 \uplus p_1$, with $\ell_i = |p_i| \geq 2$. Partition $\ell_0 + \ell_1-1$ as $[0,\ell_0-1)$ and $[\ell_0-1, \ell_0+\ell_1-1)$. Note that $\card [0,\ell_0-1) \geq 1$ and $\card [\ell_0-1, \ell_0+\ell_1-1) \geq 2$. By definition of $p_0 \uplus p_1$, for all $x \in [0, \ell_0-1)$ and $y,z \in [\ell_0-1, \ell_0 +\ell_1-1), (p_0 \uplus p_1)(x,y) = p_0(x, \ell_0-1) = (p_0 \uplus p_1)(x,z)$.

    Suppose now there is a partition $F \sqcup G$ as in the statement of the lemma. Let $p_0 = p \uh_{F \cup \{\min G\}}$ and $p_1 = p \uh_G$. We claim that $p = p_0 \uplus q_1$. First, note that $|p| = |p_0| + |p_1|-1$. For all $x<y<|p|$, we have that
    \begin{itemize}
        \item if $x<y<|p_0|$, $p(x,y)=p_0(x,y)$;
        \item if $|p_0| - 1 \leq x<y<|p_0|+|p_1|-1$, $p(x,y)=p_1(x-|p_0|+1,y-|p_0|+1)$;
        \item if $x < |p_0|-1 \leq y$, $p(x,y)= p(x,|p_0|-1)=p_0(x,|p_0|-1)$. The first equality holds since $x \in F$, $|p_0|-1, y \in G$ and by assumption on the 2-partition $F \sqcup G = \ell$.
    \end{itemize}
    This proves $p = p_0 \uplus p_1$.
\end{proof}

The following lemma states, as expected, that the join operator is associative. Therefore, we will omit the explicit parenthesis when a pattern is obtained by multiple joins. 

\begin{lemma}
    The join operator $\uplus$ is associative.
\end{lemma}

\begin{proof}
    Let $p_0,p_1,p_2$ be three $\RT_2^2$ patterns. 
    Let $q_0$ denote $(p_0 \uplus p_1) \uplus p_2$ and  $q_1$ denote $p_0\uplus (p_1 \uplus p_2)$.
    Let $x,y \leq |p_0|+|p_1|+|p_2|-2$
    \begin{itemize}
        \item If $x<y < |p_0|$, $q_0(x,y)= (p_0 \uplus p_1)(x,y)=p_0(x,y) = q_1(x,y)$;
        \item if $|p_0| \leq x <y < |p_0|+|p_1|-1$, $q_0(x,y)=(p_0 \uplus p_1)(x,y)= p_1(x-|p_0|+1,y-|p_0|+1)=q_1(x,y)$;
        \item if $|p_0|+|p_1|-1 \leq x<y < |p_0|+|p_1|+|p_2|-2$, $q_0(x,y)=p_2(x-|p_0|-|p_1|+2,y-|p_0|-|p_1|+2)= (p_1 \uplus p_2)(x -|p_1|+1,y-|p_1|+1) = q_1(x,y)$;
        \item if $ x < |p_0| \leq y < |p_0|+|p_1|-1$, then $q_0(x,y)=(p_0 \uplus p_1)(x,y) = p_0(x,|p_0|-1)=q_1(x,y)$; 
        \item if $ x < |p_0|$ and $|p_0| + |p_1|-1 \leq y$, then $q_0(x,y)=(p_0 \uplus p_1)(x,|p_0|+|p_1|-2)=p_0(x,|p_0|-1)=q_1(x,y)$;
        \item if $ |p_0| \leq x < |p_0|+|p_1|-1 \leq y$, then $q_0(x,y)=(p_0 \uplus p_1)(x, |p_0|+|p_1|-2) = p_1(x-|p_0|+1, |p_1|-1) = (p_1 \uplus p_2)(x-|p_0|+1, y-|p_0|+1) = q_1(x,y)$.
    \end{itemize}
\end{proof}

Recall that, given a pattern~$p$ of size at least~$2$, $p^-$ is its restriction to the domain~$[|p|-1]^2$. We shall see in the later sections that there exists a close relation between computing an infinite set avoiding a pattern~$p$, and computing an infinite set such that every set realizing~$p^-$ has the wrong limit. It will therefore be often useful to think of a pattern~$p$ as the pattern~$p^-$ together with a specification of a limit of the elements, given by $p(\cdot, |p|-1)$. 

\begin{definition}
A pattern~$p : [\ell]^2 \to 2$ is \emph{convergent} if there is some~$i < 2$ such that for every~$x < \ell-1$, $p(x, \ell-1) = i$. Otherwise, $p$ is \emph{divergent}.
\end{definition}

\begin{example}
The following pattern~$p$ is divergent as $p(a, c) = 1 \neq 0 = p(b, c)$. We claim that~$p$ is irreducible. Indeed, the unique 2-partition $F \sqcup G = \{a, b, c\}$ satisfying $F \neq \emptyset$, $\card G \geq 2$, and $F < G$ is $F = \{a\}$ and $G = \{ b, c \}$. But then $p(a, b) \neq p(a, c)$.
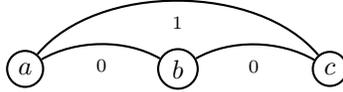
\vspace{0.5cm}
    $$
    \psmatrix[colsep=1.5cm,rowsep=1.5cm,mnode=circle]
a & b & c
\everypsbox{\scriptstyle} 
\ncarc[arcangle=30]{1,1}{1,2}_{0} 
\ncarc[arcangle=30]{1,2}{1,3}_{0}
\ncarc[arcangle=50]{1,1}{1,3}_{1} 
\endpsmatrix
$$
Actually, this pattern and its dual are the only two patterns of length~3 which are divergent and irreducible.
\end{example}


The following lemma shows that divergence is preserved among the join operator.

\begin{lemma}
    Let $p$ and $q$ be two patterns such that at least one of them is divergent. Then $p \uplus q$ is also divergent.
\end{lemma}
\begin{proof}
 Suppose $p$ is divergent, i.e., there exists $x,y < |p|$ such that $p(x,|p|-1) \neq p(y,|p|-1)$. This yields by definition of the join that $(p \uplus q)(x,|p|+|q|-2) \neq (p \uplus q)(y,|p|+|q|-2)$

 Now suppose $q$ is divergent, i.e., there exists $x,y < |q|$ such that $q(x,|q|-1) \neq q(y,|q|-1)$. This yields by definition of the join that $(p \uplus q)(x+|p|-1,|p|+|q|-2) \neq (p \uplus q)(y+|p|-1,|p|+|q|-2)$
\end{proof}

\section{$\RT^2_2$ and countable hyperimmunities}\label[section]{sect:omega-hyp}

Recall that a problem $\Psf$ preserves $\omega$ hyperimmunities if for every set~$Z$, every countable collection of $Z$-hyperimmune functions $f_0, f_1, \dots$ and every $Z$-computable $\Psf$-instance~$X$, there is a $\Psf$-solution~$Y$ to~$X$ such that every $f_i$ is $Y \oplus Z$-hyperimmune. The notion of preservation of $\omega$ hyperimmunities was used to separate the Erd\H{o}s-Moser theorem from Ramsey's theorem for pairs in reverse mathematics~\cite{lerman2013separating,patey2017iterative}. The goal of this section is to prove the following characterization theorem:

\begin{repmaintheorem}{omega-hyp-divergent-irreducible}
Let $p$ be a pattern. $\RT^2_2(p)$ preserves $\omega$ hyperimmunities if and only if $p$ contains a sub-pattern which is simultaneously divergent and irreducible.
\end{repmaintheorem}

Note that if every computable instance of a problem~$\Psf$ admits a solution of computably dominated degree, and the proof relativizes, then $\Psf$ preserves $\omega$ (and in fact even uncountably many) hyperimmunities. However, by \Cref{mainthm:dnczp}, there exists a computable coloring such that every infinite set avoiding any pattern computes a $\emptyset'$-DNC function, and by Miller (see Khan and Miller~\cite{khan_forcing_2017}, every $\emptyset'$-DNC function is of hyperimmune degree. One therefore cannot prove \Cref{omega-hyp-divergent-irreducible} by building computably dominated solutions.

The proof of \Cref{omega-hyp-divergent-irreducible} is divided into \Cref{thm:divergent-extensible-preservation-hyperimmunities} and \Cref{thm:sub-pattern-convergent-reducible-strongly-appears}.

\begin{theorem}\label[theorem]{thm:divergent-extensible-preservation-hyperimmunities}
Let $p$ be a divergent, irreducible pattern. Then $\RT^2_2(p)$ preserves $\omega$ hyperimmunities.
\end{theorem}

Note that \Cref{thm:divergent-extensible-preservation-hyperimmunities} is not strong enough, in that there are patterns which are convergent or reducible patterns~$p$ for which $\RT^2_2(p)$ preserves $\omega$ hyperimmunities.

\begin{remark}
The following pattern~$p$ is reducible, but $\RT^2_2(p)$ is equivalent to~$\EM$, and therefore preserves $\omega$ hyperimmunities.
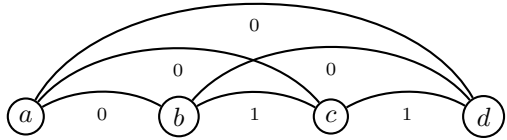
\vspace{1.5cm}
$$
\psmatrix[colsep=1.5cm,rowsep=1.5cm,mnode=circle]
a & b & c & d
\everypsbox{\scriptstyle} 
\ncarc[arcangle=30]{1,1}{1,2}_{0} 
\ncarc[arcangle=30]{1,2}{1,3}_{1}
\ncarc[arcangle=30]{1,3}{1,4}_{1}
\ncarc[arcangle=50]{1,1}{1,3}_{0} 
\ncarc[arcangle=50]{1,2}{1,4}_{0} 
\ncarc[arcangle=60]{1,1}{1,4}_{0} 
\endpsmatrix
$$
Indeed, it is  the join of the following two patterns $q$ and $r$.
\vspace{0.7cm}
$$
\psmatrix[colsep=1.5cm,rowsep=1.5cm,mnode=circle]
c & d
\everypsbox{\scriptstyle} 
\ncarc[arcangle=30]{1,1}{1,2}_{0}
\endpsmatrix
\hspace{20pt}
\psmatrix[colsep=1.5cm,rowsep=1.5cm,mnode=circle]
a & b & c
\everypsbox{\scriptstyle} 
\ncarc[arcangle=30]{1,1}{1,2}_{1} 
\ncarc[arcangle=30]{1,2}{1,3}_{1}
\ncarc[arcangle=50]{1,1}{1,3}_{0} 
\endpsmatrix
$$ 
Thus, $q$ being a sub-pattern or $r$, by \Cref{lem:avoiding-join}, $\RT^2_2(p)$ implies $\RT^2_2(r)$, which is equivalent to~$\EM$.
On the other hand, $r$ being a sub-pattern of~$p$, $\RT^2_2(r)$ implies $\RT^2_2(p)$ by \Cref{lem:sub-pattern-implication}. This example should not be considered as a flaw in the definition of irreducibility. Indeed, this simply says that avoiding~$p$ is a too weak invariant for the proof of preservation of $\omega$ hyperimmunities.
\end{remark}

The proof of \Cref{thm:divergent-extensible-preservation-hyperimmunities} is done using a variant of Mathias forcing. It requires several technical definitions and lemmas, that we now detail.
Even with non-stable colorings, every finite set of elements can be \qt{stabilized} by restricting the integers to an appropriate reservoir.

\begin{definition}
Let $E$ and $F$ be two non-empty sets such that $E < F$.
Let $f : [\NN]^2 \to 2$ be a coloring and $g : \NN \to 2$ be a partial coloring with $\dom g \supseteq E$.
We say $F$ \emph{$f$-stabilizes $E$ with witness $g$} if for all $x \in E$ and $y \in F$, $f(x,y)=g(x)$.
\end{definition}

Working with Mathias conditions for which the initial segments avoids the pattern~$p$ is not a sufficiently strong invariant to preserve hyperimmunities. We shall therefore define a stronger notion of avoidance based on the decomposition of~$p$ into $p^-$ and the function $x \mapsto p(x, |p|-1)$.


\begin{definition}
Let $p : [\ell]^2 \to 2$ be a pattern with $\ell \geq 2$ and $f : [\NN]^2 \to 2$ and $g : \NN \to 2$ be two colorings.
A set~$X$ \emph{$(f,g)$-avoids~$p$} if it $f$-avoids $p$, and for every subset~$F = \{ x_0 < \dots < x_{\ell-2}\} \subseteq X$ 
which $f$-realizes $p^{-}$, there is some~$i < \ell-1$ such that $g(x_i) \neq p(i, \ell-1)$.
\end{definition}

The following lemma relates the notion of $(f, g)$-avoidance to the notion of $f$-avoidance.
Note that if $p$ is a divergent pattern, then for every pair of colorings $f : [\NN]^2 \to 2$ and $g : \NN \to 2$,
every $g$-homogeneous set $f$-avoiding $p$ also $(f, g)$-avoids~$p$.

\begin{lemma}\label[lemma]{lem:combinatorial-fg-avoidance}
Let $p$ be a pattern. Fix two colorings $f : [\NN]^2 \to 2$ and $g : \NN \to 2$.
Let $E < F$ be two sets such that $F$ $f$-stabilizes $E$ with witness~$g$.
Then $E$ $(f, g)$-avoids~$p$ iff for every~$y \in F$, $E \cup \{y\}$ $f$-avoids~$p$.

\end{lemma}
\begin{proof}
Let $\ell = |p|$.
Suppose first that $E$ $(f, g)$-avoids~$p$, and fix $y \in F$. Suppose for the contradiction that $E \cup \{y\}$ does not $f$-avoid~$p$.
Let $H = \{ x_0, \dots x_{\ell-1} \}\subseteq E \cup \{y\}$ $f$-realize $p$. 
Since~$E$ $(f, g)$-avoids~$p$, then $E$ $f$-avoids~$p$, so $x_{\ell-1} = y$.
In particular, $H \cap E$ $f$-realizes~$p^-$, so since $E$ $(f, g)$-avoids~$p$, there is some~$i < \ell-1$ such that $g(x_i) \neq p(i,\ell-1)$. Since $F$ $f$-stabilizes~$E$ with witness~$g$, then $g(x_i) = f(x_i, x_{\ell-1})$, so $H$ does not $f$-realize~$p$.

Suppose now that $E \cup \{y\}$ $f$-avoids~$p$ for every~$y \in F$, and let $H = \{ x_0, \dots, x_{\ell -2} \} \subseteq E$ $f$-realize $p^-$. Consider $y \in F$. Since $H \cup \{y\}$ $f$-avoids $p$, there exists $i \leq \ell-2$ such that $f(x_i,y) \neq p(i,\ell-1)$, i.e., since $F$ $f$-stabilizes $E$ with witness $g$, $g(x_i)=f(x_i,y) \neq p(i,\ell-1)$. 
\end{proof}

The following lemma gives a sufficient condition to preserve $(f, g)$-avoidance by considering the union of two sets.
It is precisely where we need the irreducibility assumption in \Cref{thm:divergent-extensible-preservation-hyperimmunities}.

\begin{lemma}\label[lemma]{lem:combinatorial-extensibility}
Let $p$ be an irreducible pattern. Fix two colorings $f : [\NN]^2 \to 2$ and $g : \NN \to 2$.
Let $E < F$ be two sets such that $F$ $f$-stabilizes $E$ with witness~$g$, and $E$ and $F$ both $(f, g)$-avoid~$p$. Then $E \cup F$ $(f, g)$-avoids~$p$.
\end{lemma}
\begin{proof}
Let $\ell = |p|$. 
We first show that $E \cup F$ $f$-avoids~$p$.
Let $H = \{ x_0, \dots x_{\ell-1} \}\subseteq E \cup F$ $f$-realize $p$. 
Since~$E$ $(f, g)$-avoids~$p$, then by \Cref{lem:combinatorial-fg-avoidance}, $H \cap (E \cup \{x\})$ $f$-avoids~$p$ for every~$x \in F$, so $\card H \cap F \geq 2$.
Since~$F$ $(f, g)$-avoids~$p$, then $F$ $f$-avoids~$p$, so $H \cap E \neq \emptyset$.
Let $A = \{ i : x_i \in H \cap E \}$ and $B = \{ i : x_i \in H \cap F \}$. In particular, $A \neq \emptyset$ and $\card B \geq 2$. Since~$p$ is irreducible, then by \Cref{lem:irreducible-is-extensible}, for the 2-partition $A \sqcup B = \ell$, there is some $i \in A$ and $j, k \in B$ such that $p(i, j) \neq p(i, k)$. Since~$F$ $f$-stabilizes~$E$, then $f(x_i, x_j) = f(x_i, x_k)$, so $H$ does not $f$-realize~$p$.

We now show that for every subset~$H = \{  x_0, \dots x_{\ell-2} \} \subseteq E \cup F$ which $f$-realizes~$p^-$, there is some~$i < \ell-1$ such that $g(x_i) \neq p(i, \ell-1)$. Fix any such~$H$.
Since~$E$ and $F$ both $(f, g)$-avoid~$p$, $H \cap E$ and $H \cap F$ are both non-empty.
Let $A = \{ i : x_i \in H \cap E \}$ and $B = \ell \setminus A$.
In particular, $A \neq \emptyset$ and $\card B \geq 2$, so since~$p$ is irreducible, then by \Cref{lem:irreducible-is-extensible}, there is some~$i \in A$
and $j < k \in B$ such that $p(i, j) \neq p(i, k)$. 

Suppose $k < \ell-1$. Since $H$ $f$-realizes~$p^-$, then $f(x_i, x_j) = p(i, j)$ and $f(x_i, x_k) = p(i, k)$.
Since $F$ $f$-stabilizes~$E$ and $x_j, x_k \in F$, $f(x_i, x_j) = f(x_i, x_k)$, so $p(i, j) = p(i, k)$, contradiction.

Suppose now $k = \ell-1$. Since $F$ $f$-stabilizes~$E$ with witness~$g$, then $f(x_i, x_j) = g(x_i)$.
Since $H$ $f$-realizes~$p^-$, then $f(x_i, x_j) = p(i, j)$, so $g(x_i) = p(i, j) \neq p(i, \ell-1)$.
This completes the proof of the lemma.
\end{proof}

We are now ready to define the notion of forcing used in \Cref{thm:divergent-extensible-preservation-hyperimmunities} and study its combinatorial properties. We shall prove \Cref{thm:divergent-extensible-preservation-hyperimmunities} in a unrelativized form, so in what follows, fix a computable 2-coloring of pairs $f :[ \NN]^2 \to 2$.

\begin{definition}\label[definition]{def:forcing-notion}
   A \emph{condition} is a Mathias pair $(\sigma,X)$ such that :
\begin{itemize}
    \item $X$ $f$-stabilizes~$\sigma$ with some witness $g : \sigma \to 2$;
    \item $\sigma$ $(f, g)$-avoids~$p$;
    \item $X$ is computably dominated.
\end{itemize} 
\end{definition}

The partial order on conditions is induced by Mathias extension, that is, $(\tau,Y) \leq (\sigma,X)$ if $\sigma \preceq \tau$, $Y \subseteq X$ and $\tau \setminus \sigma \subseteq X$.
Every filter $\F$ induces a set $G_\F := \bigcup_{(\sigma,X) \in \F} \sigma$. The following lemma implies that for every sufficiently generic filter~$\F$, the set $G_\F$ is infinite.

\begin{lemma}\label[lemma]{lem:omega-hyp-generic-infinite}
    Let $c = (\sigma, X)$ be a condition and $x \in X$. There exists $Y \subseteq X$ such that $(\sigma \cup \{x\}, Y)$ is an extension of~$c$. 
\end{lemma}

\begin{proof}
    Let $Y =  X \cap (x, \infty)$ be an infinite $X$-computable subset of $X$ on which $f(x, \cdot)$ is constant. In particular, $Y$ is computably dominated since $Y \leq_T X$. We now prove that $\sigma \cup \{ x \}$ also $(f,g')$-avoids $p$, with $g'$ extending $g$ properly by letting $g(x)$ be the limit color of $f(x, \cdot)$ on~$Y$. Note that $\{ x \}$ $f$-stabilizes $\sigma$ with witness $g'$ and that since $p$ is divergent, $|p|>2$ and thus $\{x\}$ $(f,g')$-avoids~$p$. By \Cref{lem:combinatorial-extensibility}, $\sigma \cup \{ x \}$ does $(f,g')$-avoid $p$.
\end{proof}

Let $c = (\sigma,X)$ be a condition, and let $\varphi$ be a $\Sigma_1^0$ or $\Pi_1^0$ formula. We say $c$ \emph{forces} $\varphi$ denoted $c \Vdash \varphi(G)$ if for every sufficiently generic filter $\F$ containing $c$, $\varphi(G_\F)$ holds.

We shall use the forcing question framework to preserve hyperimmunities (see Patey~\cite[Chapter 3]{pateylowness}). 

\begin{definition}
Given a notion of forcing $\PP$, a \emph{forcing question} for a class for formulas $\Gamma$ with a formal set variable~$G$, is a relation $\qvdash \subseteq \PP \times \Gamma$ such that for every~$c \in \PP$ and $\varphi(G) \in \Gamma$,
\begin{itemize}
    \item[(1)] If $c \qvdash \varphi(G)$, then there is an extension $d \leq c$ forcing $\varphi(G)$;
    \item[(2)] If $c \nqvdash \varphi(G)$, then there is an extension $d \leq c$ forcing $\neg \varphi(G)$.
\end{itemize}
\end{definition}

The computability-theoretic properties of the generic set are closely related to the existence of a forcing question with good definability and combinatorial features. We shall be particularly interested in the notions of $\Sigma^0_1$-preserving and $\Sigma^0_1$-compact forcing questions.

\begin{definition}
A forcing question $\qvdash$ is \emph{$\Sigma^0_1$-preserving} if for every $\Sigma^0_1$-formula $\varphi(G, x)$ and every condition~$c$, the set $\{ x \in \NN : c \qvdash \varphi(G, x) \}$ is $\Sigma^0_1$.
\end{definition}

\begin{definition}
A forcing question $\qvdash$ is \emph{$\Sigma^0_1$-compact} if for every $\Sigma^0_1$-formula $\varphi(G, x)$,
and every condition~$c$, if $c \qvdash \exists x \varphi(G, x)$, then there is some bound~$b \in \NN$ such that $c \qvdash (\exists x \leq b)\varphi(G, x)$.
\end{definition}

The definitions above should be adapted to the actual notions of forcing we consider. For instance, when working with a Mathias-like notion of forcing, we shall say that a forcing question is $\Sigma^0_1$-preserving is the set $\{ x \in \NN : c \qvdash \varphi(G, x) \}$ is $\Sigma^0_1$ relative to the reservoir.

\begin{definition}\label[lemma]{lem:question-complete}
    Let $\exists x \varphi(G,x)$ be a $\Sigma^0_1$-formula, $c = (\sigma, X)$ be a condition. We define the forcing question $\qvdash$ as follows :
    $c \qvdash \exists x\varphi(G,x)$ if for all 2-coloring $\hat g : \N \to 2$, there exists $x \in \NN$ and a finite set $\rho \subseteq X$ which $(f,\hat g)$-avoids $p$ such that $\varphi(\sigma \cup \rho, x)$ holds.
\end{definition}

The following lemma states that the relation defined above satisfies the specifications of a forcing question for $\Sigma^0_1$-formulas:

\begin{lemma}\label[lemma]{lem:question-extension}
    Let $c$ be a condition and $\varphi$ a $\Sigma_1^0$ formula.
    \begin{itemize}
        \item If $c \qvdash \varphi(G)$, there exists $d \leq c$ such that $d \Vdash \varphi(G)$;
        \item If $c \nqvdash \varphi(G)$, there exists $d \leq c$ such that $d \Vdash \neg \varphi(G)$;
    \end{itemize}
\end{lemma}

\begin{proof}
Let $c = (\sigma, X)$.
    \begin{itemize}
        \item Suppose  $c \qvdash \varphi(G)$, hence, for all 2-partition $\hat g : \N \to 2$, there exists a finite set $\rho \subseteq X$ which $(f,\hat g)$-avoids $p$, and such that $\varphi(\sigma \cup \rho)$ holds. By a compactness argument, there exists $n \in \NN$ such that for all 2-partition $\hat g : [0, n] \to 2$, there exists a finite set $\rho \subseteq X \cap \{0, \dots, n \}$ which $(f,\hat g)$-avoids $p$, and such that $\varphi(\sigma \cup \rho)$ holds. Let $Y \subseteq X$ be an infinite $X$-computable set $f$-stabilizing $[0, n]$ with witness~$\hat g : n \to 2$. Since~$Y \leq_T X$, then $Y$ is computably dominated. Let $\rho \subseteq X \cap \{0, \dots, n \}$ be a finite set $(f,\hat g)$-avoiding $p$, and such that $\varphi(\sigma \cup \rho)$ holds. Note that $\hat g \supseteq g$, so $\sigma$ $(f, \hat g)$-avoids~$p$, and by \Cref{lem:combinatorial-extensibility}, $\sigma \cup \rho$ $(f, \hat g)$-avoids~$p$.
        In all, $d := (\sigma \cup \rho, Y)$ is a valid condition extending $c$, and is such that $d \Vdash \varphi(G)$.
        
        \item Suppose $c \nqvdash \varphi(G)$, hence, there exists a 2-partition $\hat g : \N \to 2$, such that every finite set $\rho \subseteq X$ either does not $(f,\hat g)$-avoid $p$ or is such that $\neg \varphi(\sigma \cup \rho)$ holds. Consider the $\Pi_1^0$ class of every such functions $\hat g$. By the computably dominated basis theorem (see Jockusch and Soare~\cite{jockusch197classes}), there exists such a function $\hat g$ such that $\hat g \oplus X$ is computably dominated.
        Let $Y \subseteq X$ be a $\hat g$-homogeneous and $\hat g \oplus X$-computable infinite set. Since $p$ is divergent and $Y$ is $\hat g$-homogeneous, then every set which $f$-avoids~$p$ also $(f, \hat g)$-avoids~$p$. Therefore, the condition $d := (\sigma,Y)$ forces $\neg \varphi(G)$.
    \end{itemize}
    
\end{proof}

The following lemma states that the forcing question $\qvdash$ is $\Sigma_1^0$-preserving.
\begin{lemma}\label[lemma]{lem:question-sigma01}
Let $c =(\sigma,X)$ be a condition and $\varphi$ be a $\Sigma_1^0$ formula. The sentence \qt{$c \qvdash \varphi(G)$} is $\Sigma_1^0(X)$.
\end{lemma}

\begin{proof}
    By a compactness argument, $c \qvdash \varphi(G)$ is equivalent to the $\Sigma^0_1(X)$ sentence \qt{there exists $n \in \NN$ such that for all for all 2-partition $\hat g : [0, n] \to 2$, there exists $x \in \NN$ and a finite set $\rho \subseteq X$ which $(f,\hat g)$-avoids $p$, and such that $\varphi(\sigma \cup \rho,x)$ holds.}
\end{proof}

The following lemma states that the forcing question $\qvdash$ is $\Sigma_1^0$-compact.

\begin{lemma}\label[lemma]{lem:question-compact}
Let $\exists x \, \varphi(G,x)$ be a $\Sigma_1^0$ formula and $c = (\sigma,X)$ be a condition such that $c \qvdash \exists x \varphi(G,x)$. Then, there exists  $\ell$ such that $c \qvdash \exists x \, \varphi(G,x) \land x \leq \ell$.
\end{lemma}

\begin{proof}
    As stated in the previous proof, $c \qvdash \exists x \varphi(G,x)$ implies that there is some~$n \in \NN$ such that for all for all 2-partition $\hat g : [0, n] \to 2$, there exists $x_{\hat g} \in \NN$ and a finite set $\rho \subseteq X$ which $(f,\hat g)$-avoids $p$, and such that $\varphi(\sigma \cup \rho,x_{\hat g})$ holds. Consider $\ell = \max \{ x_{\hat g}\ |\ \hat g : [0,n] \to 2 \}$. Then, for all 2-partition there exists $x \leq \ell$ and a finite set $\rho \subseteq X$ which $(f,\hat g)$-avoids $p$, and such that $\varphi(\sigma \cup \rho,x)$ holds.

\end{proof}

The following lemma is the standard diagonalization lemma which holds for every notion of forcing admitting a $\Sigma^0_1$-preserving, $\Sigma^0_1$-compact forcing question for $\Sigma^0_1$-formulas (see Patey~\cite[Chapter 3]{pateylowness}). We reprove it for the sake of completeness.

\begin{lemma}\label[lemma]{lem:preservation-hyperimmunity}
Let $c$ be a condition, $h$ be a hyperimmune function, and $e$ be a Turing index. There exists an extension $d \leq c$ such that $d \Vdash \exists x \Phi_e(G,x) \downarrow< h(x)$ or $d \Vdash \Phi_e(G) \uparrow$.
\end{lemma}

\begin{proof}

    Suppose $c \nqvdash \Phi_e(G,x) \downarrow$ for some $x \in \NN$. Then, by \Cref{lem:question-extension}, there exists $d \leq c$ such that $d \Vdash \Phi_e(G,x)\uparrow$. Now, suppose that for every $x \in \NN$, $c \qvdash \exists a \exists t \Phi_e(G,x)[t]=a$. Then, by \Cref{lem:question-compact}, for all $x$, there exists $\ell_x$ such that $c \qvdash \exists t \Phi_e(G,x)[t] \leq \ell_x$ holds. By \Cref{lem:question-sigma01}, the function $x \mapsto \ell_x$ is partial $X$-computable, and by hypothesis, it is total. Since $X$ is computably dominated, $x \mapsto \ell_x$ does not dominate $h$. This yields that there exists $x$ such that $\ell_x < h(x)$. By \Cref{lem:question-extension}, there exists $d \leq c$ such that $d \Vdash \Phi_e(G,x)\downarrow \leq \ell_x$ for that $x$, and thus $d \Vdash \Phi_e(G,x)\downarrow < h(x)$.
\end{proof}

We are now ready to prove \Cref{thm:divergent-extensible-preservation-hyperimmunities}.

\begin{proof}[Proof of \Cref{thm:divergent-extensible-preservation-hyperimmunities}]
Let $h_0, h_1, \dots$ be a countable sequence of hyperimmune functions and $f : [\NN]^2 \to 2$ be a computable coloring. 
Let $\F$ be a sufficiently generic filter for the associated notion of forcing. By \Cref{lem:omega-hyp-generic-infinite}, the set $G_\F$ is infinite, and by definition of a forcing condition, $G_\F$ $f$-avoids~$p$. Moreover, by \Cref{lem:preservation-hyperimmunity}, $h_i$ is $G_\F$-hyperimmune for every~$i \in \NN$.
This completes the proof of \Cref{thm:divergent-extensible-preservation-hyperimmunities}.
\end{proof}

\begin{corollary}\label[corollary]{cor:sub-pattern-preserves-omega-hyperimmunities}
Let $p$ be a pattern containing as a sub-pattern a divergent, irreducible pattern.
Then $\RT^2_2(p)$ preserves $\omega$ hyperimmunities.
\end{corollary}
\begin{proof}
Immediate by \Cref{thm:divergent-extensible-preservation-hyperimmunities} and the fact that if $q$ is a sub-pattern of~$p$, then every $\RT^2_2(q)$-solution is an $\RT^2_2(p)$-solution.
\end{proof}

The remainder of this section is devoted to the proof of the reciprocal, that is, if a pattern~$p$ does not contains any sub-pattern which is simultaneously divergent and irreducible, then $\RT^2_2(p)$ does not preserve $\omega$ hyperimmunities simultaneously. Actually, we shall prove that $\RT^2_2(p)$ does not even preserve 2 hyperimmunities simultaneously (\Cref{cor:sub-pattern-not-2-hyperimmunity}).

For this, we need a stronger notion of appearance, which does not only state that the pattern~$p$ appears in the set, but that $p^-$ appears with the right limit, in the sense defined below. A coloring $f : [\NN]^2 \to 2$ is \emph{stable} if for every~$x \in \NN$, $\lim_y f(x, y)$ exists.

\begin{definition}
Let $f : [\NN]^2 \to 2$ be a stable coloring and $p : [\ell]^2 \to 2$ be a pattern of size~$\ell$.
We say that a finite set $F = \{ x_0 < \dots < x_{\ell-2} \}$ \emph{strongly $f$-realizes} $p$ if $F$ $f$-realizes~$p^-$ and for every $i < \ell-1$ and all but finitely many~$y \in \NN$, $f(x_i, y) = p(i, \ell-1)$. 
We say the pattern $p$ \emph{strongly $f$-appears} in a set $H$ if there exists a finite subset $F \subseteq H$ which strongly $f$-realizes $p$.
\end{definition}

\begin{remark}\label[remark]{rem:strongly-appears-implies-appears}
Note that if $H$ is infinite, and $p$ strongly $f$-appears in $H$, then $p$ $f$-appears in $H$. Indeed, consider a finite set $R$ strongly $f$-realizing $p$ in $H$ (such exists by definition of strong $f$-appearance), and let $t \in H$ be such that every element of $R$ has reached its $f$-limit from $t$ on, $t$ included (such $t$ exists by infinity of $H$). Then $R \cup \{ t \}$ $f$-realizes $p$.
\end{remark}

\begin{theorem}\label[theorem]{thm:sub-pattern-convergent-reducible-strongly-appears}
    Let $p$ be a pattern with $|p| \geq 2$, such that all of its sub-patterns are convergent or reducible, and let $A$ be a $\Delta_2^0$ bi-hyperimmune infinite set. Let $f : [\N]^2 \to 2$ be a  $\Delta_2^0$ approximation of $A$. For every infinite set $H$ such that $A$ and $\overline{A}$ are $H$-hyperimmune, 
    the pattern $p$ strongly $f$-appears in~$H$.
\end{theorem}

\begin{proof}
    Let us prove the desired result by induction on the size of~$p$. Let $H$ be an infinite set such that $A$ and $\overline{A}$ are $H$-hyperimmune. 
    
    First, suppose that $p$ is convergent. If $|p| = 2$, then trivially, for every~$k \in \NN$, $p^-$ $f$-appears in~$H \cap (k, \infty)$. If $|p| > 2$, then by induction hypothesis, for every~$k \in \NN$, $p^-$ strongly $f$-appears in $H \cap (k, \infty)$, so by \Cref{rem:strongly-appears-implies-appears}, $p^-$ $f$-appears in $H \cap (k, \infty)$. Since $f$ is computable, one can find $H$-computably an infinite array $F_0, F_1, \dots \subseteq H$ such that each block $f$-realizes $p^-$. Now, suppose w.l.o.g. that $p(0,|p|-1)=0$. Then, by $H$-hyperimmunity of $A$ there exists some~$n \in \NN$ such that such that $F_n \subseteq \overline{A}$. Since~$f$ is a $\Delta^0_2$-approximation of~$A$,  for all $x \in F_n$, $\lim_{y \in \NN} f(x,y) = 0$. This yields that $F_n$ strongly $f$-realizes~$p$, so $p$ strongly appears in $H$.

    Suppose now that $p$ is not convergent. By assumption, $p$ is then reducible. As such, by \Cref{lem:irreducible-is-extensible} there exists a partition $F \cup G$ of $|p|$ such that $\card F > 0$, $\card G \geq 2$, $F < G$, and such that for all $x \in F$, and $\forall y,z \in G$, $p(x,y) = p(y,z)$. Let $z = \min G$. By induction hypothesis, $p \uh_{F \cup \{ z \}}$ strongly appears in $H$ and for every $k \in \NN$, $p \uh_G$ strongly appears in~$H \cap (k, \infty)$. Let $R = \{ x_0, \dots, x_{|F|-1} \} \subseteq H$ be strongly $f$-realizing $p \uh_{F \cup \{ z \}}$. Let $k$ be large enough so that all elements of~$R$ have reached their limit and $S = \{ x_{|F|}, \dots, x_{|p|-1} \} \subseteq H \cap (k, \infty)$ be strongly $f$-realizing $p \uh_G$.

    We first claim that $R \cup S$ $f$-realizes $p^- = p \uh_{F \cup G^-}$. Let $i < j \in F \cup G^-$, where $G^- = G \setminus \{\max G\}$. If $j \in F$, then since $R$ $f$-realizes $p \uh_F$, $f(x_i, x_j) = p(i, j)$. If $i \in G^-$, then since $S$ $f$-realizes $p \uh_{G^-}$, $f(x_i, x_j) = p(i, j)$. If $i \in F$ and $j \in G^-$, then since $R$ strongly $f$-realizes $p \uh_{F \cup \{ z \}}$, $f(x_i,x_j)= p(i,z)$. However, since $p$ is reducible, $p(i,z)=p(i,j)$, so $f(x_i, x_j) = p(i, j)$. It follows that $R \cup S$ $f$-realizes $p^-$.
    
    We now claim that $R \cup S$ strongly $f$-realizes $p$. Let $i \in F \cup G^-$. If $i \in F$, since $p$ is reducible, $p(i, z) = p(i, |p|-1)$. Since $R$ strongly $f$-realizes $p \uh_{F \cup \{ z \}}$, for cofinitely many~$y$, $f(x_i, y) = p(i, z) = p(i, |p|-1)$. If $i \in G^-$, since $S$ strongly $f$-realizes $p \uh_G$, for cofinitely many $y$, $f(x_i,y) = p(i, |p|-1)$. It follows that $R \cup S$ strongly $f$-realizes $p$.
\end{proof}

\begin{corollary}\label[corollary]{cor:sub-pattern-not-2-hyperimmunity}
Let $p$ be a pattern with $|p| \geq 2$, such that all of its sub-patterns are convergent or reducible. Then $\RT^2_2(p)$ does not preserve 2 hyperimmunities, as witnessed by a stable coloring.
\end{corollary}
\begin{proof}
Suppose for the contradiction that $\RT^2_2(p)$ preserves 2 hyperimmunities. Let $A$ be a $\Delta^0_2$ bi-hyperimmune set, of $\Delta^0_2$-approximation $f : [\NN]^2 \to 2$, and let $H$ be an infinite set $f$-avoiding $p$ and such that $A$ and $\overline{A}$ are both $H$-hyperimmune. By \Cref{rem:strongly-appears-implies-appears}, $p$ does not strongly $f$-appear in $H$, so by \Cref{thm:sub-pattern-convergent-reducible-strongly-appears}, $p$ contains a sub-pattern which is both divergent and irreducible.
\end{proof}

We are now ready to prove \Cref{omega-hyp-divergent-irreducible}.
\begin{proof}[Proof of \Cref{omega-hyp-divergent-irreducible}]
Suppose first $p$ contains a divergent and irreducible sub-pattern. Then by \Cref{cor:sub-pattern-preserves-omega-hyperimmunities}, $\RT^2_2(p)$ preserves $\omega$ hyperimmunities. 
Suppose now that $p$ does not contain such a sub-pattern. Then by \Cref{cor:sub-pattern-not-2-hyperimmunity}, $\RT^2_2(p)$ does not preserve 2 (and a fortiori $\omega$) hyperimmunities.
\end{proof}

\section{Preservation of 2-dimensional hyperimmunity}\label[section]{sect:2dim-hyp}

In this section, we prove a similar characterization theorem for a variant of hyperimmunity called 2-dimensional hyperimmunity.
As mentioned in \Cref{sec:preservation-immunities}, this variant might seem much more ad-hoc than hyperimmunity, but it is arguably the natural combinatorial notion obtained by designing an invariant property not preserved by the Erd\H{o}s-Moser theorem. It is defined and  successfully used by Liu and Patey~\cite{liu2022reverse} to prove that the free set theorem does not imply $\EM$ over~$\omega$-models.
In an follow-up paper, Le Houérou and Patey~\cite{houerou2025ramseylike} proved that if a Ramsey-like theorem does not preserve $\omega$ hyperimmunities, then it implies $\RT^2_2$ over $\omega$-models. We conjecture that, similarly, if a Ramsey-like theorem does not preserve one 2-dimensional hyperimmunity, then it implies $\EM$ over $\omega$-models.

\begin{definition}
    A \emph{bi-family} is a collection $\Hc$ of ordered pairs of finite sets closed under subset product, i.e., if $(A, B) \in \Hc$ and $C \subseteq A$ and $D \subseteq B$, then $(C,D) \in \Hc$.
    A \emph{bi-array} is a collection of finite sets $(\vec{E}, \vec{F}) = \langle E_n, F_{n,m} : n,m, \in \NN \rangle$ such that $\min E_n > n$, $\min F_{n,m} > m$ for every $n,m \in \NN$. A bi-array $(\vec{E}, \vec{F})$ \emph{meets} a bi-family $\Hc$ if there is some $n,m \in \NN$ such that $(E_n, F_{n,m}) \in \Hc$.
    A bi-family $\Hc$ is \emph{$2$-dimensional $C$-hyperimmune} if every $C$-computable bi-array meets $\Hc$.
\end{definition}

\begin{remark}
There exists a natural generalization of the previous definition to every dimension. In dimension 1, a 1-array is nothing but a traditional c.e.\ array. Then, if $A$ is an co-hyperimmune set, the collection $\Hc$ of all finite subsets of~$A$ is a 1-family which is 1-dimensional hyperimmune. The notion of $n$-dimensional hyperimmunity is therefore rather a generalization of the notion of co-hyperimmunity than of hyperimmunity.
\end{remark}

This variant of hyperimmunity induces a family of notions of preservation:

\begin{definition}
Fix $k \in \NN \cup \{\NN\}$.
A problem~$\Psf$ \emph{preserves $k$ 2-dimensional hyperimmunities} if for every set~$Z$,
every family of 2-dimensional $Z$-hyperimmune bi-families  $\langle \Hc_s : s < k \rangle$, and every $Z$-computable $\Psf$-instance~$X$, there exists a $\Psf$-solution~$Y$ to~$X$ such that each~$\Hc_s$ is 2-dimensional $Y \oplus Z$-hyperimmune.    
\end{definition}

The goal of this section is therefore to characterize the patterns~$p$ for which $\RT^2_2(p)$ preserves 1 2-dimensional hyperimmunity. For this, we need to define the following combinatorial property of a pattern:

\begin{definition}
A pattern $p : [\ell]^2 \to 2$ is \emph{$i$-merging} if for every non trivial 2-partition $F \sqcup G = \ell-1$
such that $F < G$, one of the following holds:
\begin{enumerate}
    \item $\exists x \in F$, $p(x, \ell-1) = 1-i$
    \item $\exists x\in G$, $p(x, \ell-1) = i$
    \item $\exists x_0, x_1 \in F\ \exists y_0, y_1 \in G$, $p(x_0, y_0) \neq p(x_1, y_1)$.
\end{enumerate}
\end{definition}

The notion of divergent pattern can be seen as being $i$-merging for some $i<2$ for the degenerate case of trivial partitions. Indeed, 
if one considers the two partitions of $\ell-1$ as $\emptyset \sqcup \ell-1$ and $\ell-1 \sqcup \emptyset$, then there must exist both $x \in \ell-1$ such that $p(x, \ell-1) = i$ and $x \in \ell-1$ such that $p(x, \ell-1) = 1-i$.
Also note that every pattern $p : [\ell]^2 \to 2$  is $i$-merging for some $i<2$, more specifically, it is $1-p(0,\ell-1)$-merging and $p(\ell-2,\ell-1)$-merging.

\begin{lemma}
    Every convergent pattern $p$ of size at least 3 is both $0$-merging and $1$-merging.
\end{lemma}

\begin{proof}
    Say w.l.o.g. that $p$ is convergent such that $p(0,|p|-1)=p(|p|-2,|p|-1)=0$. Then, for all non-trivial partition $F \sqcup G$ of $\ell-1$, $0 \in F$ and as such $p$ is $1$-merging, and $|p|-2 \in G$ and as such $p$ is $0$-merging.
\end{proof}

The reciprocal of the previous lemma does not hold.
The remainder of this section is devoted to the proof of the following theorem:

\begin{theorem}\label[theorem]{thm:2dimhyper}
    Let $p$ be a pattern. $\RT_2^2(p)$ preserves one 2-dimensional hyperim\-munity if and only if $p$ contains two divergent and irreducible sub-patterns $p_0$ and $p_1$ such that $p_0$ is $0$-merging and $p_1$ is $1$-merging.
\end{theorem}

Note that $p_0$ and $p_1$ are not necessarily distinct, in which case we shall see in the next section that $\RT^2_2(p)$ preserves $\omega$ 2-dimensional hyperimmunities. As for \Cref{omega-hyp-divergent-irreducible}, the proof of \Cref{thm:2dimhyper} is divided into \Cref{thm:disjunctive-2-dim-hyp-preserves} and \Cref{thm:2-dim-reciprocal}.

\begin{theorem}\label[theorem]{thm:disjunctive-2-dim-hyp-preserves}
Let $p_0$ and $p_1$ be irreducible, divergent, and respectively $0$-merging and $1$-merging patterns, and let $\Hc$ be a 2-dimensional hyperimmune bi-family.
For every computable coloring $f : [\NN]^2 \to 2$, there is an infinite set~$H$ $f$-avoiding $p_i$ for some~$i < 2$ and such that $\Hc$ is 2-dimensional $H$-hyperimmune.
\end{theorem}

The proof of \Cref{thm:disjunctive-2-dim-hyp-preserves} is done using a disjunctive variant of Mathias forcing. Before defining the notion of forcing, we prove a technical lemma which gives a sufficient condition to take two sets $E$ and $F$, and obtain a union $E \cup F$ which $(f, g)$-avoids~$p$. In some sense, \Cref{lem:combinatorial-merging-for-color} is similar to \Cref{lem:combinatorial-extensibility}. The main difference is that $g$ is not assumed to witness the fact that $F$ $f$-stabilizes $E$. Because of this, 
$p$, $E$ and $F$ must satisfy stronger hypothesis.

\begin{lemma}\label[lemma]{lem:combinatorial-merging-for-color}
Let $p$ be a pattern divergent and $i$-merging for some color~$i < 2$. Fix two colorings $f : [\NN]^2 \to 2$ and $g : \NN \to 2$.
Let $E < F$ be two sets such that 
\begin{itemize}
    \item[(1)] $E$ is $g$-homogeneous; $F$ is $g$-homogeneous for color~$1-i$;
    \item[(2)] For every $x_0, x_1 \in E$, for every $y_0, y_1 \in F$, $f(x_0, y_0) = f(x_1, y_1)$;
    \item[(3)] $E \cup F$ $f$-avoids~$p$.
\end{itemize}
Then $E \cup F$ $(f, g)$-avoids~$p$.
\end{lemma}
\begin{proof}

Let $\ell = |p|$, and let $H = \{ x_0, \dots, x_{\ell-2} \}\subseteq E \cup F$ $f$-realize $p^-$.
We have two cases:

Case 1: $H$ is $g$-homogeneous. Since $p$ is divergent, there is some~$i$ such that $p(i, \ell-1) \neq g(x_i)$. Since by (3) $E \cup F$ $f$-avoids $p$, then $E \cup F$ $(f, g)$-avoids~$p$.

Case 2: $H$ is not $g$-homogeneous. Then in particular $E$ is $g$-homogeneous for color~$i$.
Since $p$ is $i$-merging , one of the following holds:

\begin{itemize}
    \item $\exists x_i \in E \cap H$ such that $p(i, \ell-1) = 1-i$. Since $E$ is $g$-homogeneous for color~$i$, $p(i, \ell-1) \neq g(x_i)$.
    \item $\exists x_i \in F \cap H$ such that $p(i, \ell-1) = i$. Since $F$ is $g$-homogeneous for color~$1-i$, $p(i, \ell-1) \neq g(x_i)$.
    \item $\exists x_0, x_1 \in E \cap H$, $\exists y_0, y_1 \in F \cap H$ such that $p(x_0, y_0) \neq p(x_1, y_1)$.
    This third case contradicts item (2) in the hypothesis, so this does not happen.
\end{itemize}
Since for every $H  = \{ x_0, \dots, x_{\ell-2} \} \subseteq E \cup F$ $f$-realizing $p^-$, there is some~$i$ such that $p(i, \ell-1) \neq g(x_i)$ and $E \cup F$ $f$-avoids~$p$, then $E \cup F$ $(f, g)$-avoids~$p$.
\end{proof}

We shall again prove \Cref{thm:disjunctive-2-dim-hyp-preserves} in an unrelativized form.
In what follows, fix a 2-dimensional hyperimmune bi-family $\Hc$, and a computable coloring $f : [\NN]^2 \to 2$.
We shall construct two infinite sets $G_0, G_1$, such that
$G_i$ $f$-avoids $p_i$ for each~$i < 2$ and $\Hc$ is 2-dimensional hyperimmune relative to either $G_0$ or $G_1$.

\begin{definition}
A \emph{condition} is a 3-tuple $(\sigma_0,\sigma_1,X)$ where, for both $i<2$, \begin{itemize}
    \item $(\sigma_i,X)$ is a Mathias condition;
    \item $X$ $f$-stabilizes~$[0, \max(\sigma_0, \sigma_1)]$ with some witness $g : [0, \max(\sigma_0, \sigma_1)] \to 2$;
    \item $\sigma_i$ $(f, g)$-avoids~$p_i$;
    \item $\Hc$ is 2-dimensional $X$-hyperimmune. 
\end{itemize}
\end{definition}

The order over conditions is defined as follows : $(\sigma_0,\sigma_1,X) \leq (\tau_0,\tau_1,Y)$ if $Y \subseteq X$ and for every~$i < 2$, $\sigma_i \preceq \tau_i$ and $\tau_i \setminus \sigma_i \subseteq X$. Every sufficiently generic filter~$\F$ induces two sets $G_{\F, 0}$ and $G_{\F, 1}$, defined by $G_{\F, i} = \bigcup_{(\tau_0,\tau_1,Y) \in \F}\tau_i$.

Given an arithmetic formula $\varphi$, we write $c \Vdash \Phi(G_i)$ if for every sufficiently generic filter $\F$ containing $c$, $\varphi(G_{\F, i})$ holds. The following lemma states that if $\F$ is a sufficiently generic filter, then $G_{\F,0}$ and $G_{\F, 1}$ are both infinite.

\begin{lemma}\label[lemma]{lem:generic-infinite}
    Let $c = (\sigma_0,\sigma_1,X)$ be a condition and $x_0,x_1 \in X$. There exists $Y \subseteq X$ such that $(\sigma_0 \cup \{x_0\}, \sigma_1 \cup \{x_1\}, Y)$ is an extension of~$c$. 
\end{lemma}

\begin{proof}
    Let $Y =  X \cap (\max(x_0, x_1), \infty)$ be an infinite $X$-computable subset of $X$ which $f$-stabilizes $[0, \max(x_0, x_1)]$ with some witness~$\hat g$. In particular, $\Hc$ is still 2-dimensional hyperimmune relative to $Y$. We now prove that for both $i$, $\sigma_i \cup \{ x_i \}$ also $(f,\hat g)$-avoids $p_i$. Note that $\{ x_i \}$ $f$-stabilizes $\sigma_i$ with witness $\hat g$ and that since $p_i$ is divergent, $|p_i|>2$ and thus $\{x_i\}$ $(f,\hat g)$-avoids $p_i$. By \Cref{lem:combinatorial-extensibility}, $\sigma_i \cup \{ x_i \}$ does $(f,\hat g)$-avoid $p_i$. The condition $(\sigma_0 \cup \{ x_0 \},\sigma_1 \cup \{ x_1 \} ,Y)$ is the desired extension of~$c$.
\end{proof}

We are now going to define two kind of forcing questions for $\Sigma^0_1$-formulas: a non-disjunctive one $\qvdash_i$ for each side~$i < 2$ (\Cref{def:question-i}) and a disjunctive one $\qvdash$ (\Cref{def:disjunctive-question}). Both kinds will be shown to be $\Sigma^0_1$-preserving and $\Sigma^0_1$-compact. They additionally satisfy a merging lemma (\Cref{lem:merging}) enabling to prove the diagonalization lemma (\Cref{lem:generic-hyperimmune}).

\begin{definition}\label[definition]{def:question-i}
Let $\varphi$ be a $\Sigma^0_1$-formula, $c = (\sigma_0, \sigma_1, X)$ be a condition and $i < 2$. We define the forcing question $\qvdash_i$ as follows :
    $c \qvdash_i \varphi(G)$ if for all pairs of 2-colorings $h_0 : \N \to 2$ and $h_1 : \N \to 2$, there exists a finite $h_0$-homogeneous and $h_1$-homogeneous set $\rho \subseteq X$ which $(f,h_0)$-avoids $p_i$ such that $\varphi(\sigma_i \cup \rho)$ holds.
\end{definition}

\begin{remark}
    In the previous definition, we asked for $\rho$ to be $h_0$-homogeneous and to $(f,h_0)$-avoid $p_i$. One could note that since $p_i$ is asked to be divergent, being $h_0$-homogeneous almost already yields $(f,h_0)$-avoiding $p_i$, as long as $\rho$ $f$-avoids $p_i$. This means this definition could be a little bit weaker. However, we will mostly use the fact $\rho$ $(f,h_0)$-avoids $p_i$, so it simplifies argumentation to directly ask for it in the definition.
\end{remark}

\Cref{def:question-i} satisfies the specification of a forcing question, with a proof similar to that of \Cref{lem:question-extension}.
We shall actually prove a stronger version of the specifications through \Cref{lem:merging}.
For now, we prove that the forcing question is $\Sigma^0_1$-preserving.

\begin{lemma}\label[lemma]{lem:question-i-preserving}
    Let $c =(\sigma_0,\sigma_1,X)$ be a condition and $\varphi$ be a $\Sigma_1^0$ formula. The sentence \qt{$c \qvdash_i \varphi(G)$} is $\Sigma_1^0(X)$ for both $i<2$.
\end{lemma}

\begin{proof}
   Let $\varphi \equiv \exists x \psi(G, x)$ be a $\Sigma_1^0$ condition. By a compactness argument, $c \qvdash_i \varphi(G)$ holds if and only if there exists $t \in \NN$ such that for all 2-colorings $h_0, h_1 : t \to 2$, there exists $\rho \subseteq X \cap \{0, \dots, t \}$ and $x < t$ such that $\rho$ is $h_0$-homogeneous, $h_1$-homogeneous, $(f,h_0)$-avoids $p_i$ and such that $\psi(\sigma_i \cup \rho, x)$ holds.
\end{proof}

We now define a notion of $\Sigma^0_1$-compactness for forcing question. Recall that $\Sigma^0_1$-compactness of a forcing question means that if a condition answers positively to the question for a $\Sigma_1^0$ formula, then we can actually bound the first existential quantifier. 

\begin{lemma}\label[lemma]{lem:question-i-compact}
    Fix $i < 2$ and a condition $c$. For every $\Delta^0_0$-formula $\varphi(G, x)$, if $c \qvdash_i \exists x \varphi(G,x)$, then there exists $t \in \NN$ such that $c \qvdash_i \exists x \leq t \ \varphi(G,x)$.
\end{lemma}

\begin{proof}
Let $t$ be the bound given by the proof of \Cref{lem:question-i-preserving}. Then $c \qvdash_i \exists x \leq t \ \varphi(G,x)$ holds.
\end{proof}

We now define the disjunctive forcing question for pairs of $\Sigma^0_1$-formulas.

\begin{definition}\label[definition]{def:disjunctive-question}
    Let $\varphi_0(G)$ and $\varphi_1(G)$ be two $\Sigma^0_1$-formulas and $c = (\sigma, X)$ be a condition. Let $c \qvdash \varphi_0(G) \vee \varphi_1(G)$ if for all 2-coloring $h : \N \to 2$, there exists a side $i < 2$ and a finite set $\rho \subseteq X$ which $(f,h)$-avoids $p_i$ such that $\varphi_i(\sigma_i \cup \rho)$ holds.
\end{definition}

Again, we prove that the disjunctive forcing question is $\Sigma^0_1$-preserving and $\Sigma^0_1$-compact.

\begin{lemma}\label[lemma]{lem:question-preserving}
    Let $c =(\sigma_0,\sigma_1,X)$ be a condition and $\varphi_0(G)$ and $\varphi_1(G)$ be two $\Sigma^0_1$-formulas. The sentence \qt{$c \qvdash \varphi_0(G) \vee \varphi_1(G)$} is $\Sigma^0_1(X)$.
\end{lemma}

\begin{proof}
   Let $\varphi_0(G) \equiv \exists x \psi_0(G, x)$ and $\varphi_1(G) \equiv \exists x \psi_1(G, x)$ be $\Sigma_1^0$ formulas. By a compactness argument, $ c \qvdash_i \varphi_0(G_0) \lor \varphi_1(G_1)$ holds if and only if there exists $t \in \NN$ such that for all 2-coloring $h : t \to 2$, there exists $i<2$, $x < t$ and $\rho \subseteq X \cap \{0, \dots, t \}$ which $(f,h)$-avoids $p_i$ and such that $\psi_i(\sigma_i \cup \rho, x)$ holds.
\end{proof}

\begin{lemma}
    Let $c$ be a condition and $\varphi_0(G,x)$ and $\varphi_1(G,x)$ be two $\Delta^0_0$-formulas. If $c \qvdash \exists x \varphi_0(G_0,x) \lor \exists x \varphi_0(G_0,x)$, then there exists $t \in \NN$ such that $c \qvdash \exists x<t \,\varphi_0(G_0,x) \lor \exists x<t \, \varphi_0(G_0,x)$.
\end{lemma}

\begin{proof}
Again, the bound~$t$ in the proof of \Cref{lem:question-preserving} is such that $c \qvdash \exists x<t \,\varphi_0(G_0,x) \lor \exists x<t \, \varphi_0(G_0,x)$ holds. 
\end{proof}

We now prove the core merging lemma satisfied by the forcing questions. In particular, the first item says that one can find a simultaneous witness to a negative answer from the disjunctive forcing question and a positive answer from the non-disjunctive forcing questions.

\begin{lemma}\label[lemma]{lem:merging}
    Let $\varphi_0,\varphi_1,\psi_0,\psi_1$ be $\Sigma^0_1$ formulas, and let $c$ be a condition. 

    \begin{itemize}
        \item If  $c \qvdash_i \varphi_i(G_i)$ for both $i<2$ and $c \nqvdash \psi_0(G_0) \lor \psi_1(G_1)$,
    then there exists $d \leq c$ such that $d \Vdash \varphi_i(G_i) \wedge \neg \psi_i(G_i)$ for some $i<2$;
        \item If $c \nqvdash_i \varphi_i(G_i)$ for some $i<2$, then there exists $d \leq c$ such that $d \Vdash \neg \varphi_i(G_i)$;
        \item If $c \qvdash \psi_0(G_0) \lor \psi_1(G_1)$, then there exists  $d \leq c$ such that $d \Vdash \psi_0(G_0) \lor \psi_1(G_1)$.
    \end{itemize}

\end{lemma}

\begin{proof}
Say $c = (\sigma_0, \sigma_1, X)$.
\begin{itemize}
    \item  Since $c \nqvdash \psi_0(G_0) \lor \psi_1(G_1)$, the class~$\C$ of every 2-coloring witnessing that failure is non-empty. Note that $\C$ is a $\Pi_1^0(X)$-class, and by \cite[Corollary 2.7]{liu2022reverse}, $\WKL$ preserves one 2-dimensional hyperimmunity, so there exists an element $h \in \C$ such that $\Hc$ is still 2-dimensional hyperimmune relative to~$h \oplus X$. 

    Since for both $i<2$, $c \qvdash_i \varphi_i(G)$, by compactness, there exists some~$\ell \in \NN$ such that  for both $i<2$, the following property $(\dagger_i)$ holds: for every pair of partitions $h_0 : [0, \ell] \to 2$ and $h_1 : [0, \ell] \to 2$, there exists a finite $h_0$-homogeneous and $h_1$-homogeneous set $\rho_i \subseteq X \cap [0, \ell]$ which $(f,h_0)$-avoids $p_i$ such that $\varphi_i(\sigma_i \cup \rho_i)$ holds.
    
    Let $Y \subseteq X \setminus [0, \ell]$ be an infinite $X \oplus h$-computable subset which is $h$-homogeneous and $f$-stabilizes $[0, \ell]$, say with witness $g' : [0, \ell] \to 2$. 
    Let $i < 2$ be such that $Y$ is $h$-homogeneous for color~$1 -i$. Without loss of generality, suppose $i=0$. By $(\dagger_0)$, letting $h_0 = g'$ and $h_1 = h$, there is a finite $g'$ and $h$-homogeneous set $\rho \subseteq X \cap [0, \ell]$ which $(f,g')$-avoids $p_0$ such that $\varphi_0(\sigma_0 \cup \rho)$ holds.
    
    Note that $\rho$ and $Y$ are both $h$-homogeneous, but not necessarily of the same color.

    We first claim that $d = (\sigma_0 \cup \rho, \sigma_1, Y)$ is a valid condition. Indeed, by \Cref{lem:combinatorial-extensibility}, $\sigma_0 \cup \rho$ $(f, g')$-avoids $p_0$. Moreover, by choice of~$Y$, it $f$-stabilizes~$\sigma_0 \cup \rho$ with witness~$g'$. Last, $Y$ is $X \oplus h$-computable, hence $\Hc$ is 2-dimensional hyperimmune relative to~$Y$.

    We now claim that $d \Vdash \neg \psi_i(G_i)$ for both~$i < 2$. Indeed, let $\tau \subseteq Y$ be such that $\sigma_i \cup \rho \cup \tau$ $(f, g')$-avoids~$p_i$. In particular, $\rho \cup \tau$ $f$-avoids $p_i$ and $\rho$ is also $h$-homogeneous for color $1-i$, thus, by \Cref{lem:combinatorial-merging-for-color}, $\rho \cup \tau$ $(f, h)$-avoids $p_i$, hence $\neg \psi_i(\sigma_i \cup \rho \cup \tau)$ holds since  $h \in \C$.

    Last, $d \vdash \varphi_0(G_0)$ by choice of~$\rho$.

    \item 
    Suppose $c \nqvdash_i \varphi_i(G)$, and consider the class~$\C$ of every pair of 2-colorings witnessing that failure. Note that $\C$ is a $\Pi_1^0(X)$-class, and by \cite[Corollary 2.7]{liu2022reverse}, $\WKL$ preserves one 2-dimensional hyperimmunity : there exists an element $(h_0,h_1) \in \C$ such that $\Hc$ is still 2-dimensional hyperimmune relative to $X \oplus h_0 \oplus h_1$. Let $a,b<2$ be such that $Y := X \cap \{ x : h_0(x)=a$ and $h_1(x)=b \}$ is infinite. The condition $d := (\sigma_0,\sigma_1,Y)$ is a valid condition below $c$. We claim that $d \Vdash \neg \varphi_i(G_i)$ : for any extension $(\tau_0,\tau_1,Z)$ of $d$, by construction, letting $\rho = \tau_i \setminus \sigma_i$,  $\rho \subseteq Y$, and as such $\rho$ is both $h_0$-homogeneous and $h_1$-homogeneous. Moreover, $\tau_i$ hence $\rho$ $f$-avoids $p_i$, and since $p$ is divergent, $\rho$ does $(f,h_i)$-avoid $p_i$. Thus $\neg \varphi_i(\tau_i)$ will hold.

    \item Suppose $c \qvdash \psi_0(G_0) \lor \psi_1(G_1)$. By compactness, there exists some~$\ell \in \NN$ such that the following property $(\dagger)$ holds: for every coloring $h : [0, \ell] \to 2$, there exists a side~$i < 2$ and a finite set $\rho \subseteq X \cap [0, \ell]$ which $(f, h)$-avoids $p_i$ such that $\psi_i(\sigma_i \cup \rho)$ holds.

    Let $Y \subseteq X \setminus [0, \ell]$ be an infinite $X$-computable subset $f$-stabilizing $[0, \ell]$, say with witness $g' : [0, \ell] \to 2$. By $(\dagger)$, letting $h = g'$, there is some~$i < 2$ and a finite set~$\rho \subseteq X \cap [0, \ell]$ which $(f, g')$-avoids~$p_i$ and such that $\psi_i(\sigma_i \cup \rho)$ holds.

    Without loss of generality, suppose $i=0$. We first claim that $d = (\sigma_0 \cup \rho, \sigma_1, Y)$ is a valid condition. Indeed, by \Cref{lem:combinatorial-extensibility}, $\sigma_0 \cup \rho$ $(f, g')$-avoids $p_0$. Moreover, by choice of~$Y$, it $f$-stabilizes~$\sigma_0 \cup \rho$ with witness~$g'$. Last, $Y$ is $X \oplus h$-computable, hence $\Hc$ is 2-dimensional hyperimmune relative to~$Y$.

    Last, $d \Vdash \psi_0(G_0)$ by choice of~$\rho$.
\end{itemize}

\end{proof}

We have all the necessary tools to prove the diagonalization lemma for preservation of a 2-dimensional hyperimmunity, using the existence of $\Sigma^0_1$-preserving, $\Sigma^0_1$-compact forcing questions satisfying the merging lemma (\Cref{lem:merging}).

\begin{lemma}\label[lemma]{lem:generic-hyperimmune}
    Let $\Hc$ be a 2-dimensional hyperimmune bi-family, $c$ be a condition and $\Phi_{e_0}$, $\Phi_{e_1}$ be two 2-array functionals. There exists an extension~$d$ of $c$ forcing either $\Phi_{e_i}^{G_i}$ to be partial or $\Hc$ to intersect $\Phi_{e_i}^{G_i}$ for some $i<2$.
\end{lemma}

\begin{proof}
Say $c = (\sigma_0,\sigma_1,X)$.
    \begin{itemize}
        \item Let $\psi(n)$ be a set $F_n$ such that $c \qvdash_i \Phi_{e_i}^{G_i}(n)\downarrow \subseteq F_n$ for each~$i < 2$, if it exists. Note that by \Cref{lem:question-i-preserving}, the relation $c \qvdash_i \Phi_{e_i}^{G_i}(n)\downarrow \subseteq F_n$ is $\Sigma_1^0(X)$ uniformly in~$n$, so $\psi(n)$ is also $\Sigma_1^0(X)$ uniformly in~$n$.
        \item Let $\psi(n;m)$ be a set $F_{n,m}$ such that 
        $$c \qvdash \bigvee_{i < 2} (\Phi_{e_i}^{G_i}(n)\downarrow \subseteq \psi(n) \land \Phi_{e_i}^{G_i}(n,m)\downarrow \subseteq F_{n,m})$$
        if it exists. Note that by \Cref{lem:question-preserving}, the relation above is $\Sigma_1^0(X)$ uniformly in~$n, m$, hence $\psi(n;m)$ is also $\Sigma_1^0(X)$ uniformly in~$n, m$.
    \end{itemize}

    Now, three cases can hold :

    \begin{itemize}

        \item \emph{Case 1 :} $\exists n \, \psi(n) \uparrow$. Then, by \Cref{lem:question-i-compact}, there is some~$i < 2$ such that $c \nqvdash_i \Phi_{e_i}^{G_i}(n)\downarrow$, and by \Cref{lem:merging}, there is some extension~$d \leq c$ forcing $\Phi_{e_i}^{G_i}(n) \uparrow$.

        \item \emph{Case 2:} $\exists n,m \ (\psi(n)\downarrow \land\ \psi(n;m) \uparrow)$. Then, by \Cref{lem:merging}, there is an extension~$d \leq c$ forcing $\Phi_{e_i}^{G_i}(n) \downarrow \subseteq F_n$ for some $i<2$ and forcing $\Phi_{e_i}^{G_i}(n) \downarrow \subseteq F_n \implies \Phi_{e_i}^G(n,m)\uparrow$, hence $d$ forces 
        $\Phi_{e_i}^{G_i}(n,m)\uparrow$.

        \item \emph{Case 3:} $\forall n,m \, (\psi(n) \downarrow \land\ \psi(n;m) \downarrow)$. Then, by 2-dimensional hyperimmunity of $\Hc$ relative to~$X$, there exist $n,m$ such that $(\psi(n),\psi(n;m)) \in \Hc$.  Moreover, by \Cref{lem:merging}, for some $i<2$, there is an extension~$d \leq c$ forcing $\Phi_{e_i}^{G_i}(n) \downarrow \subseteq \psi(n) \land \Phi_{e_i}^{G_i}(n,m)\downarrow \subseteq \psi(n;m)$, i.e., $(\Phi_{e_i}^{G_i}(n), \Phi_{e_i}^{G_i}(n,m)) \in \Hc$

    \end{itemize}
\end{proof}

We are now ready to prove \Cref{thm:disjunctive-2-dim-hyp-preserves}.

\begin{proof}[Proof of \Cref{thm:disjunctive-2-dim-hyp-preserves}]
Let $\Hc$ be a 2-dimensional hyperimmune bi-family, and $f : [\NN]^2 \to 2$ be a computable coloring.
Consider a sufficiently generic filter $\F$ for the associated notion of forcing. For both $i<2$, by \Cref{lem:generic-infinite},  $G_{\F,i}$ is an infinite set, and, by definition of a condition, $G_{\F,i}$ $f$-avoids $p_i$. Finally, by \Cref{lem:generic-hyperimmune}, $\Hc$ is $2$-hyperimmune relative to $G_{\F,i}$ for some $i<2$.
This completes the proof of \Cref{thm:disjunctive-2-dim-hyp-preserves}.
\end{proof}

\begin{corollary}\label[corollary]{cor:irreducible-divergent-merging-subpattern-preserves}
Let $p$ be a pattern containing two irreducible and divergent sub-patterns $p_0$ and $p_1$ such that for each~$i < 2$, $p_i$ is merging for color~$i$.
Then $\RT^2_2(p)$ preserves 2-dimensional hyperimmunity.
\end{corollary}
\begin{proof}
Immediate by \Cref{thm:disjunctive-2-dim-hyp-preserves} and \Cref{lem:sub-pattern-implication}.
\end{proof}

The goal is now to prove the reciprocal of \Cref{cor:irreducible-divergent-merging-subpattern-preserves}. For this, we need to construct some specific 2-dimensional hyperimmune bi-family~$\Hc$ based on a stable coloring $f : [\NN]^2 \to 2$, such that for every infinite set~$H$ avoiding the desired pattern~$p$, $\Hc$ is not 2-dimensional $H$-hyperimmune. Recall that a coloring $f : [\NN]^2 \to 2$ is stable if for every~$x$, $\lim_y f(x, y)$ exists.

\begin{definition}
Fix a stable coloring $f : [\NN]^2 \to 2$. Given two sets $E < F$, we write $E \to_i F$ for $(\forall x \in E)(\forall y \in  F)f(x, y) = i$. For every $i < 2$, we let $A_i(f) = \{x : (\forall^\infty y)f(x, y) = i\}$. Finally, we let $\Hc_i(f)$ be the bi-family of all pairs $(E, F)$ such that $E < F$, $E \subseteq A_i(f), F \subseteq A_{1-i}(f), $ and $E \to_{1-i} F$. 
\end{definition}

The following result and its proof are an immediate adaptation of Liu and Patey~\cite[Proposition 2.10]{liu2022reverse}.

\begin{theorem}\label[theorem]{thm:f-2-dim-hyperimmune}
There exists a stable computable coloring $f : [\NN]^2 \to 2$ such that for each~$i < 2$, $\Hc_i(f)$ is 2-dimensional hyperimmune.
\end{theorem}
    
\begin{proof}
    We build the coloring $f : [\N]^2 \to 2$ by a finite injury priority argument. For
every $e \in \omega$, we want to satisfy the following requirement:
\begin{quote}
    $\R_{e,i}$ : If $\Phi_e$ is total, then there is some $n, m \in \N$ such that 
$\Phi_e(n) \subseteq A_i(f)$, $\Phi_e(n; m) \subseteq A_{1-i}(f)$ and $\Phi_e(n) \to_{1-i} \Phi_e(n; m)$.
\end{quote}

The requirements are given the usual priority ordering $\R_{0,0} < \R_{0,1} < \R_{1,0} \dots$. Initially, the
requirements are neither partially, nor fully satisfied.
\begin{itemize}
    \item A requirement $\R_{e,i}$ \emph{requires a first attention} at stage $s$ if it is not partially
satisfied and $\Phi_{e}(n)[s] \downarrow= E$ for some set $E \subseteq \{e + 1, \dots , s - 1\}$ such that
no element in E is restrained by a requirement of higher priority. If it
receives attention, then it puts a restraint on $E$, commit the elements of $E$
to be in $A_{1-i}(f)$, and is declared partially satisfied.
    \item  A requirement $\R_{e,i}$ \emph{requires a second attention} at stage $s$ if it is not fully
satisfied, $\Phi_{e}(n)[s] \downarrow= E$ and $\Phi_{e}[s](n; m) \downarrow= F$ for some sets $E < F \subseteq
\{e + 1, \dots , s - 1\}$ such that $E \to_{1-i} F$ and which are not restrained by
a requirement of higher priority. If it receives attention, then it puts a
restraint on $E \cup F$, commits the elements of $E$ to be in $A_{i}(f)$, the elements
of $F$ to be in $A_{1-i}(f)$, and is declared fully satisfied.
\end{itemize}
At stage $0$, we let $f = \emptyset$. Suppose that at stage $s$, we have defined $f(x, y)$ for every
$x < y < s$. For every $x < s$, if it is committed to be in some $A_i(f)$, set $f(x, s) = i$. Let $\R_{e,i}$ be the requirement of highest priority which
requires attention. If $\R_{e,i}$ requires a second attention, then execute the second
procedure, otherwise execute the first one. In any case, reset all the requirements
of lower priorities by setting them unsatisfied, releasing all their restraints, and go
to the next stage. This completes the construction. One easily sees by induction that each requirement $\R_{e,i}$ acts finitely often, and either $\Phi_e$ is partial, in which case $\R_{e,i}$ is vacuously satisfied, or is eventually fully satisfied. This
procedure also yields a stable coloring. 
\end{proof}

\begin{remark}\label[remark]{rem:2-dim-hyp-is-hyp}
Let $f : [\NN]^2 \to 2$ be a stable computable coloring such that for some~$i < 2$, $\Hc_i(f)$ is 2-dimensional hyperimmune and let $A_i = \{ x : \lim_y f(x, y) = i \}$. Then $A_0$ and $A_1$ are both hyperimmune. Indeed, any c.e.\ array $(E_n : n \in \NN)$ can be transformed into a bi-array $(E_n, F_{n,m} : n, m \in \NN)$ by letting $F_{n,m} = E_m$. If $\Hc_i(f)$ is 2-dimensional hyperimmune, then $(E_n, F_{n,m}) \in \Hc_i(f)$ for some~$n, m \in \NN$, in which case $E_n \subseteq A_i$ and $E_m \subseteq A_{1-i}$, so both $A_0$ and $A_1$ are hyperimmune.
\end{remark}

Recall that a finite set $F = \{ x_0 < \dots < x_{\ell-2} \}$ \emph{strongly $f$-realizes} a pattern $p$ if $F$ $f$-realizes~$p^-$ and for every $i < \ell-1$ and all but finitely many~$y \in \NN$, $f(x_i, y) = p(i, \ell-1)$. 
Accordingly, a pattern $p$ \emph{strongly $f$-appears} in a set $H$ if there exists a finite subset $F \subseteq H$ which strongly $f$-realizes $p$.


\begin{theorem}\label[theorem]{thm:2-dim-reciprocal}
    Fix $i < 2$.
    Let $p$ be a pattern of size at least~2 such that all of its sub-patterns are either reducible, convergent, or non-$i$-merging, and let $f : [\N]^2 \to 2$ be a stable computable coloring. For every infinite set $H$ such that $\Hc_i(f)$ is 2-dimensional $H$-hyperimmune, the pattern $p$ strongly $f$-appears in~$H$.
\end{theorem}

\begin{proof}
    Let us prove the desired result by induction on the size of $p$. Let $H$ be an infinite set such that $\Hc_0(f)$ is 2-dimensional $H$-hyperimmune.

    First, suppose that $p$ is convergent. If $|p| = 2$, then for every~$k \in \NN$, the singleton pattern $p^-$ $f$-appears in~$H \cap (k, \infty)$. If $|p| > 2$, then by induction hypothesis, for every~$k \in \NN$, $p^-$ strongly $f$-appears in $H \cap (k, \infty)$, so by \Cref{rem:strongly-appears-implies-appears}, $p^-$ $f$-appears in $H \cap (k, \infty)$. Since $f$ is computable, one can find $H$-computably an infinite array $\{F_n\}_{n\in \NN}$ included in~$H$ and such that each $F_n$ $f$-realizes $p^-$. 
    By \Cref{rem:2-dim-hyp-is-hyp}, $A_0$ and $A_1$ are both $H$-hyperimmune, so there is some~$n$ such that $F_n \subseteq A_{p(0,|p|-1)}$. Then $F_n$ strongly $f$-realizes~$p$, so $p$ strongly $f$-appears in $H$.

    Suppose now that $p$ is divergent and reducible. By the same argument as in \Cref{thm:sub-pattern-convergent-reducible-strongly-appears}, $p$ strongly $f$-appears in~$H$. 
    

    Finally, suppose $p$ divergent, irreducible, and not $i$-merging. This yields a non-trivial partition $F \cup G = \ell-1$ such that $F < G$ and such that :
 \begin{enumerate}
    \item $\forall x \in F$, $p(x, \ell-1) = i$
    \item $\forall x\in G$, $p(x, \ell-1) = 1-i$
    \item $\forall x_0,x_1 \in F\ \forall y_0, y_1 \in G$, $p(x_0, y_0) = p(x_1, y_1)$.
\end{enumerate}

If the color of item~3 is $i$, then $p = p \uh_{F \cup \{\min G\}} \uplus\ p \uh_{G \cup \{\ell-1\}}$, contradicting the fact that $p$ is irreducible, so assume $\forall x \in F \forall y \in G\ p(x, y) = 1-i$. By induction hypothesis, for every $k \in \NN$, $p \uh_F$ and $p \uh_G$ strongly $f$-appear in $H \cap (k, \infty)$. In particular, by \Cref{rem:strongly-appears-implies-appears}, for every~$k \in \NN$, $p \uh_F$ and $p \uh_G$ $f$-appear in $H \cap (k, \infty)$.

Consider an $H$-computable bi-array $(R_n, S_{n,m} : n,m > 0)$ such that $R_n \subseteq H$ $f$-realizes $p \uh_F$ and $S_{n,m} \subseteq H$ $f$-realizes $p \uh_G$.  Since $\Hc_i(f)$ is 2-dimensional $H$-hyperimmune, $(R_n, S_{n,m} : n,m > 0)$ intersects $\Hc_i(f)$, in other words, there exists $n,m \in \NN$ such that $(R_n, S_{n,m}) \in \Hc_i(f)$. By definition of $\Hc_i(f)$, $R_n \subseteq A_i$, $S_{n,m} \subseteq A_{1-i}$ and $R_n \to_{1-i} S_{n,m}$.
By item (3), $R_n \cup S_{n,m}$ $f$-realize $p^-$, and by items (1-2), the set $R_n \cup S_{n,m} \subseteq H$ strongly $f$-realizes $p$.
\end{proof}

\begin{proof}[Proof of \Cref{thm:2dimhyper}]
Suppose first $p$ contains two divergent and irreducible sub-patterns $p_0$ and $p_1$ such that $p_0$ is $0$-merging and $p_1$ is $1$-merging. Then by \Cref{cor:irreducible-divergent-merging-subpattern-preserves}, $\RT_2^2(p)$ preserves one 2-dimensional hyperimmunity.

Conversely, suppose that $\RT_2^2(p)$ preserves one 2-dimensional hyperimmunity.
Let $f$ be a stable function such that both $\Hc_0(f)$ and  $\Hc_1(f)$ are 2-dimensional hyperimmune. 
Fix some~$i < 2$ and let $H$ be an infinite set $f$-avoiding $p$ such that $\Hc_i(f)$ is 2-dimensional $H$-hyperimmune. By \Cref{rem:strongly-appears-implies-appears}, $p$ does not strongly $f$-appear in $H$, so by \Cref{thm:2-dim-reciprocal}, $p$ contains a sub-pattern which is divergent, irreducible and $i$-merging.
\end{proof}

\section{Preservation of $\omega$ $2$-dim hyperimmunities}\label[section]{sect:omega-2dim-hyp}

Patey~\cite{patey2017iterative} proved the existence, for every~$k$, of a problem $\Psf_k$ which preserves $k$, but not $k+1$ hyperimmunities. Since $\RT^2_2$ preserves one hyperimmunity, so do every statement $\RT^2_2(p)$ where $p$ is a non-constant pattern. However, the proof of \Cref{omega-hyp-divergent-irreducible} showed that if $\RT^2_2(p)$ does not preserve $\omega$ hyperimmunities, then it does not preserve 2 hyperimmunities either, so the hierarchy collapses for this family of statements. This collapsing is actually due to the fact that we consider only 2-colorings. Indeed, the statement $\Psf_k$ studied by Patey~\cite{patey2017iterative} is of the form $\RT^2_{k+1}(p)$ for some pattern~$p : [\ell]^2 \to k+1$.

In the case of 2-dimensional hyperimmunities, the hierarchy also collapses at level 2 for 2-colorings. In this section, we characterize the statements $\RT^2_2(p)$ which preserve $\omega$ 2-dimensional hyperimmunities in terms of~$p$. As it turns out, the characterization is very similar to \Cref{thm:2dimhyper}, except that one requires the existence of a single sub-pattern which is simultaneously 0-merging and 1-merging. The existence of such a sub-pattern enables to define a non-disjunctive notion of forcing, and as such, to satisfy all the requirements independently.


\begin{definition}
A pattern $p : [\ell]^2 \to 2$ is \emph{merging} if it is both $0$-merging and $1$-merging.
\end{definition}

The goal is to prove the following theorem:

\begin{theorem}\label[theorem]{thm:main-non-disjunctive-2-dim-hyp-preserves}
Let $p$ be a pattern. $\RT^2_2(p)$ preserves $\omega$ 2-dimensional hyperimmunities iff $p$ contains a sub-pattern which is simultaneously irreducible, merging and divergent.
\end{theorem}

As usual, the proof of \Cref{thm:main-non-disjunctive-2-dim-hyp-preserves} is divided into \Cref{thm:non-disjunctive-2-dim-hyp-preserves} and \Cref{thm:omega-2-dim-reciprocal}. All the proofs are straightforward adaptations of \Cref{sect:2dim-hyp} to a non-disjunctive setting, so they will be omitted.

\begin{theorem}\label[theorem]{thm:non-disjunctive-2-dim-hyp-preserves}
Let $p$ be an irreducible, divergent, and merging pattern. Then $\RT^2_2(p)$ preserves $\omega$ 2-dimensional hyperimmunities.
\end{theorem}

Fix a countable collection 2-dimensional hyperimmune bi-families $\Hc_0, \Hc_1, \dots$, and a computable coloring $f : [\NN]^2 \to 2$.
We shall construct an infinite set $G$ such that $G$ $f$-avoids $p$ and for every~$n$, $\Hc_n$ is 2-dimensional $G$-hyperimmune. 

\begin{definition}
A \emph{condition} is a 2-tuple $(\sigma,X)$ where \begin{itemize}
    \item $(\sigma,X)$ is a Mathias condition;
    \item $X$ $f$-stabilizes~$[0, |\sigma|-1]$ with some witness $g : \sigma \to 2$;
    \item $\sigma$ $(f, g)$-avoids~$p$;
    \item $\Hc_n$ is 2-dimensional $X$-hyperimmune for every~$n \in \NN$.
\end{itemize}
\end{definition}
The order over conditions is the usual Mathias extension.
Every sufficiently generic filter~$\F$ induces a set $G_\F = \bigcup_{(\tau,Y) \in \F}\tau$.
The following standard lemma states that for every sufficiently generic filter~$\F$, the set $G_\F$ is infinite.

\begin{lemma}\label[lemma]{lem:nd-generic-infinite}
    Let $c = (\sigma,X)$ be a condition and $x \in X$. There exists $Y \subseteq X$ such that $(\sigma \cup \{x\},  Y)$ is an extension of~$c$. 
\end{lemma}

We now define two non-disjunctive, $\Sigma^0_1$-preserving and $\Sigma^0_1$-compact forcing questions, which satisfy a non-disjunctive version of the merging lemma (\Cref{lem:nd-merging}). As mentioned, we only state the definitions and lemmas without proofs.

\begin{definition}
Let $\varphi$ be a $\Sigma^0_1$-formula, $c = (\sigma, X)$ be a condition. We define the forcing question $\qvdash'$ as follows :
    $c \qvdash' \varphi(G)$ if for all pairs of 2-colorings $h_0 : \N \to 2$ and $h_1 : \N \to 2$, there exists a finite $h_0$-homogeneous and $h_1$-homogeneous set $\rho \subseteq X$ which $(f,h_0)$-avoids $p$ such that $\varphi(\sigma \cup \rho)$ holds.
\end{definition}

\begin{lemma}\label[lemma]{lem:nd-question-i-preserving}
Let $c =(\sigma,X)$ be a condition and $\varphi$ be a $\Sigma_1^0$ formula. The sentence \qt{$c \qvdash' \varphi(G)$} is $\Sigma_1^0(X)$.
\end{lemma}

\begin{lemma}\label[lemma]{lem:nd-question-i-compact}
    For every condition $c$ and every $\Delta^0_0$-formula $\varphi(G,x)$, if $c \qvdash' \exists x \varphi(G,x)$, then there exists $t \in \NN$ such that $c \qvdash' \exists x \leq t \ \varphi(G,x)$.
\end{lemma}

\begin{definition}
    Let $\varphi(G)$ be a $\Sigma^0_1$-formula and $c = (\sigma, X)$ be a condition. We define the forcing question $\qvdash$ as follows :
    $c \qvdash \varphi(G,x)$ if for all 2-coloring $h : \N \to 2$, there exists a finite set $\rho \subseteq X$ which $(f,h)$-avoids $p$ such that $\varphi(\sigma \cup \rho)$ holds.
\end{definition}

\begin{lemma}\label[lemma]{lem:nd-question-preserving}
Let $c =(\sigma,X)$ be a condition and $\varphi$ be a $\Sigma_1^0$ formula. The sentence \qt{$c \qvdash \varphi(G)$} is $\Sigma_1^0(X)$.
\end{lemma}

\begin{lemma}
    For every condition $c$ and every $\Delta^0_0$-formula $\varphi(G,x)$, if $c \qvdash \exists x \varphi(G,x)$, then there exists $t \in \NN$ such that $c \qvdash \exists x<t \,\varphi(G,x) $.
\end{lemma}

The following merging lemma serves the same purpose as in every previous section and chapter.
Simply note that Liu and Patey~\cite[Corollary 2.7]{liu2022reverse} actually proved that $\WKL$ preserves $\omega$ 2-dimensional hyperimmunities, so the computability-theoretic constraint on the reservoirs of a conditions can be preserved with the same combinatorics.

\begin{lemma}\label[lemma]{lem:nd-merging}
    Let $\varphi$ and $\psi$ be two $\Sigma^0_1$ formulas, and let $c \in \PP$. 

    \begin{itemize}
        \item If  $c \qvdash' \varphi(G)$ and $c \nqvdash \psi(G)$,
    then there exists $d \leq c$ such that $d$ forces $\varphi(G) \wedge \neg \psi(G)$ ;
        \item If $c \nqvdash' \varphi(G)$, then there exists $d \leq c$ such that $d \mbox{ forces } \neg \varphi(G)$;
        \item If $c \qvdash \psi(G)$, then there exists  $d \leq c$ such that $d \mbox{ forces }  \psi$.
    \end{itemize}
\end{lemma}

Finally, the existence of two $\Sigma^0_1$-preserving, $\Sigma^0_1$-compact forcing questions satisfying the previous merging lemma enables to prove the following diagonalization lemma:

\begin{lemma}\label[lemma]{lem:nd-generic-hyperimmune}
    Let $\Hc$ be a 2-dimensional hyperimmune bi-family, $c =(\sigma,X) \in \PP$ and $\Phi_{e}$ be a 2-array functional. There exists an extension~$d$ of $c$ forcing either $\Phi_{e}^{G}$ to be partial or $\Hc$ to intersect $\Phi_{e}^{G}$.
\end{lemma}

As mentioned, the notion of forcing being non-disjunctive, the requirements for preserving each 2-dimensional hyperimmunity of each bi-family can be satisfied independently, without resorting to a pairing argument.
We are now ready to prove \Cref{thm:non-disjunctive-2-dim-hyp-preserves}.

\begin{proof}[Proof of \Cref{thm:non-disjunctive-2-dim-hyp-preserves}]
Fix a countable collection of 2-dimensional hyperimmune bi-families $\Hc_0, \Hc_1, \dots$ and a computable coloring $f : [\NN]^2 \to 2$.
Consider a sufficiently generic filter $\F$ for the associated notion of forcing. By \Cref{lem:nd-generic-infinite},  $G_\F$ is an infinite set, and, by definition of a condition $G_\F$ $f$-avoids $p$. Finally, by \Cref{lem:nd-generic-hyperimmune}, $\Hc_n$ is $2$-dimensional $G$-hyperimmune for every~$n \in \NN$.
This completes the proof of \Cref{thm:non-disjunctive-2-dim-hyp-preserves}.
\end{proof}


\begin{corollary}\label[corollary]{cor:nd-irreducible-divergent-merging-subpattern-preserves}
Let $p$ be a pattern containing an irreducible, merging and divergent sub-pattern.
Then $\RT^2_2(p)$ preserves $\omega$ 2-dimensional hyperimmunities.
\end{corollary}

The following theorem proves the reciprocal in a strong sense: if $\RT^2_2(p)$ does not preserve $\omega$ 2-dimensional hyperimmunities, then it does not even preserve 2 of them.

\begin{theorem}\label[theorem]{thm:omega-2-dim-reciprocal}
Let $p$ be a pattern of size at least~2 such that all of its sub-patterns are either reducible, convergent, or non-merging, and let $f : [\N]^2 \to 2$ be a stable computable coloring. For every infinite set $H$ such that $\Hc_0(f)$ and $\Hc_1(f)$ are both 2-dimensional $H$-hyperimmune, the pattern $p$ strongly $f$-appears in~$H$.
\end{theorem}
\begin{proof}
The proof is exactly the same as the one of \Cref{thm:2-dim-reciprocal}, except that in the case analysis, if $p$ is not merging, then it is not $i$-merging for some~$i < 2$, and one exploits 2-dimensional $H$-hyperimmunity of $\Hc_i(f)$. Because the choice of~$i$ depends on the considered sub-pattern, both $\Hc_0(f)$ and $\Hc_1(f)$ must be 2-dimensional $H$-hyperimmune.
\end{proof}

\begin{corollary}\label[corollary]{cor:omega-2-dim-clean-reciprocal}
Let $p$ be a pattern of size at least~2 such that all of its sub-patterns are either reducible, convergent, or non-merging. Then $\RT^2_2(p)$ does not preserve two 2-dimensional hyperimmunities, as witnessed by a stable coloring.
\end{corollary}
\begin{proof}
Suppose for the contradiction that $\RT_2^2(p)$ preserves two 2-dimensional hyperimmunities.
Let $f$ be a stable computable function such that both $\Hc_0(f)$ and  $\Hc_1(f)$ are 2-dimensional hyperimmune. Such a coloring exists by \Cref{thm:f-2-dim-hyperimmune}.
Let $H$ be an infinite set $f$-avoiding $p$ such that $\Hc_0(f)$ and $\Hc_1(f)$ are 2-dimensional $H$-hyperimmune. By \Cref{rem:strongly-appears-implies-appears}, $H$ strongly $f$-avoids~$p$, so by \Cref{thm:omega-2-dim-reciprocal}, $p$ contains a sub-pattern which is divergent, irreducible and merging.   
\end{proof}

\begin{proof}[Proof of \Cref{thm:main-non-disjunctive-2-dim-hyp-preserves}]
Suppose first $p$ contains a divergent, merging and irredu\-cible sub-patterns $q$. Then by \Cref{cor:nd-irreducible-divergent-merging-subpattern-preserves}, $\RT_2^2(p)$ preserves $\omega$ 2-dimensional hyperimmunities.
Suppose now $p$ does not contain such a sub-pattern. Then by \Cref{cor:omega-2-dim-clean-reciprocal}, it does not preserve 2 (and a fortiori $\omega$) 2-dimensional hyperim\-munities.
\end{proof}

\section{The Half Erd\H{o}s-Moser theorem}\label[section]{sec:hem}

There exist multiple known decompositions of $\RT^2_2$ in combinatorially simpler statements.
Cholak, Jockusch and Slaman~\cite{cholak_jockusch_slaman_2001} decomposed $\RT^2_2$ into its stable version ($\SRT^2_2$) and the cohesiveness principle ($\COH$). Then, Bovykin and Weiermann~\cite{bovykin2017strength} split $\RT^2_2$ into the Erd\H{o}s-Moser theorem ($\EM$) and the Chain AntiChain principle ($\CAC$), and Mont\'alban noticed that $\ADS$, which is strictly weaker than~$\CAC$, was actually sufficient. The Chain AntiChain principle  states, for every partial order on~$\NN$, the existence of an infinite chain or antichain.
Both $\ADS$ and $\CAC$ can be formulated in terms of transitivity. 

\begin{proposition}[{Hirschfeldt and Shore~\cite[Section 5]{Hirschfeldt2007CombinatorialPW}}]Over $\RCA_0$,
\begin{itemize}
    \item $\ADS$ is equivalent to the statement \qt{Every coloring $f : [\NN]^2 \to 2$ which is transitive for both colors admits an infinite homogeneous set}
    \item $\CAC$ is equivalent to the statement \qt{Every coloring $f : [\NN]^2 \to 2$ which is transitive for one color admits an infinite homogeneous set}
\end{itemize}
\end{proposition}

Thus, given a coloring $f : [\NN]^2 \to 2$, $\EM$ states the existence of an infinite set $H \subseteq \NN$ on which $f$ is transitive for both colors, and $\ADS$ applied to $f \uh [H]^2 \to 2$ yields an infinite $f$-homogeneous set. There exists a natural counterpart to this decomposition, involving $\CAC$ and an asymmetric version of the Erd\H{o}s-Moser theorem.

\begin{definition}[Half Erd\H{o}s-Moser theorem]
    The statement $\HEM$ is the following: \qt{For every 2-coloring of pairs $f : [\NN]^2 \to 2$, there exists an infinite set transitive for at least one color.}
\end{definition}

This statement is designed to obtain the following decomposition.

\begin{proposition}
$\RCA_0 \vdash \RT^2_2 \leftrightarrow (\HEM \wedge \CAC)$.
\end{proposition}
\begin{proof}
Let $f : [\NN]^2 \to 2$ be an instance of $\RT^2_2$.
By $\HEM$, there is an infinite set $X = \{ x_0 < x_1 < \dots \}$ and some color~$i < 2$ such that $X$ is $f$-transitive for color~$i$.
Let $g : [\NN]^2 \to 2$ be defined by $g(a, b) = f(x_a, x_b)$.
In particular, $g$ is transitive for color~$i$, so by Hirschfeldt and Shore~\cite[Theorem 5.2]{Hirschfeldt2007CombinatorialPW}, $\CAC$ proves the existence of an infinite $g$-homogeneous set $Y \subseteq \NN$.
The set $\{ x_a : a \in Y \}$ is $f$-homogeneous.
\end{proof}

This decomposition is arguably slightly less natural than the one in terms of $\EM$ and $\ADS$, but is interesting from the viewpoint of the first-order part of $\RT^2_2$.

The \emph{first-order part} of a second-order theory~$T$ is the set of all the first-order sentences provable by~$T$. Understanding the first-order part of theorems is an important part of the reverse mathematical process, as it is informative of the strength of a statement and closely related to the vision of reverse mathematics as a partial realization of Hilbert's program~\cite{simpson_partial_1988}. In particular, the quest for the first-order part of Ramsey's theorem for pairs is a very active branch of reverse mathematics. It is known to strictly follow from $\Sigma_2$-induction ($\isig_2$) and to imply the $\Sigma_2$-collection scheme ($\bsig_2$). See Cholak, Jockusch and Slaman~\cite{cholak_jockusch_slaman_2001} for the former result, and Hirst~\cite{hirst1987combinatorics} for the latter one.

The decomposition of $\RT^2_2$ in terms of $\HEM$ and $\CAC$ is particularly interesting, as $\CAC$ is the strongest known consequence of $\RT^2_2$ for which the first-order part is known to be equivalent to $\bsig_2$ (see Chong, Slaman and Yang~\cite{chong2021pi11}). By an amalgamation theorem of Yokoyama~\cite{yokoyama_pi_11_2010}, it follows that the first-order part of $\RT^2_2$ is $\bsig_2$ iff it is the case for $\HEM$. We therefore devote this section to a better understanding of the reverse mathematical strength of this statement.

First of all, thanks to \Cref{lem:avoiding-join}, $\HEM$ can be casted in the Ramsey-like framework and is of the form $\RT^2_2(p_0 \uplus p_1)$ where $p_0$ and $p_1$ are the non-transitivity patterns for color~0 and~1 (see \Cref{fig:non-transitivity}). This makes $\HEM$ benefit from the general analysis of Ramsey-like theorems above. In particular, since $p_0 \uplus p_1$ is a pattern of standard size, we obtain from \Cref{prop:rca-bsig2-rt22p-2dnc} the following lower bound. Recall that $\DNCS{n}$ is the statement \qt{For every set~$X$, there is an $X^{(n-1)}$-DNC function}.

\begin{proposition}
$\RCA_0 + \BSig_2 \vdash \HEM \to \DNCS{2}$.
\end{proposition}

It is however unknown whether $\HEM$ implies $\BSig_2$ over $\RCA_0$,
as the known proof of $\RCA_0 \vdash \EM \to \BSig_2$ by Kreuzer~\cite{kreuzer2012primitive} produces a coloring $f : [\NN]^2 \to 2$ which is transitive for some color, hence which is trivial from the viewpoint of~$\HEM$.

\begin{proposition}\label[proposition]{prop:hem-preserves}
    $\HEM$ preserves 2-dimensional hyperimmunity.
\end{proposition}

\begin{proof}
Given $i < 2$, let $p_i$ be the non-transitivity pattern for color~$i$. It is divergent, irreducible, and $i$-merging. Let $p = p_0 \uplus p_1$. By construction, both $p_0$ and $p_1$ are sub-patterns of $p$. As such, by \Cref{thm:2dimhyper}, $\RT_2^2(p)$ preserves 2-dimensional hyperimmunity. Finally, together with \Cref{lem:avoiding-join}, this yields that $\HEM$ preserves 2-dimensionnal hyperimmunity.
\end{proof}

\begin{corollary}
    $\WKL + \HEM + \COH$ does not imply $\EM$ over~$\RCA_0$.
\end{corollary}
\begin{proof}
By \Cref{prop:hem-preserves}, and Liu and Patey~\cite[Corollary 2.7, Corollary 2.9-]{liu2022reverse} $\HEM$, $\WKL$ and $\COH$ preserve 2-dimensional hyperimmunity, while by Liu and Patey~\cite[Corollary 2.12]{liu2022reverse},  $\EM$ does not, as witnessed by a stable coloring.
\end{proof}

\begin{proposition}\label[proposition]{prop:hem-does-not-preserve}
    $\HEM$ does not preserve two 2-dimensional hyperimmunities, as witnessed by a stable coloring.
\end{proposition}
\begin{proof}
 As mentioned, $\HEM$ is equivalent over $\RCA_0$ to $\RT^2_2(p_0 \uplus p_1)$, where $p_0$ and $p_1$ are as in \Cref{fig:non-transitivity}. The pattern $p_0 \uplus p_1$ does not contain any sub-pattern which is simultaneously irreducible, merging and divergent, so by \Cref{cor:omega-2-dim-clean-reciprocal}, $\HEM$ does not preserve two 2-dimensional hyperimmunities.
\end{proof}

\begin{corollary}
$\WKL + \COH$ does not imply $\HEM$ over $\RCA_0$.
\end{corollary}
\begin{proof}
By Liu and Patey~\cite[Corollary 2.7, Corollary 2.9]{liu2022reverse}, both $\WKL$ and $\COH$ preserve two 2-dimensional hyperimmunities over $\RCA_0$, while $\HEM$ does not by \Cref{prop:hem-does-not-preserve}.
\end{proof}

\begin{center}
\textbf{Acknowledgment}
\end{center}
The authors are thankful to the anonymous referee for his numerous comments and improvements suggestions.

\bibliographystyle{plain}
\bibliography{biblio}

@book{simpson_2009, place={Cambridge}, edition={2}, series={Perspectives in Logic}, title={Subsystems of Second Order Arithmetic}, DOI={10.1017/CBO9780511581007}, publisher={Cambridge University Press}, author={Simpson, Stephen G.}, year={2009}, collection={Perspectives in Logic}}

@article{cholak_jockusch_slaman_2001, title={On the strength of Ramsey's theorem for pairs}, volume={66}, DOI={10.2307/2694910}, number={1}, journal={The Journal of Symbolic Logic}, publisher={Cambridge University Press}, author={Cholak, Peter A. and Jockusch, Carl G. and Slaman, Theodore A.}, year={2001}, pages={1–55}}

@article{Hirschfeldt2007CombinatorialPW,
  title={Combinatorial principles weaker than Ramsey's Theorem for pairs},
  author={Denis R. Hirschfeldt and Richard A. Shore},
  journal={Journal of Symbolic Logic},
  year={2007},
  volume={72},
  pages={171 - 206}
}

@article {chong2021pi11,
    AUTHOR = {Chong, C. T. and Slaman, Theodore A. and Yang, Yue},
     TITLE = {{$\Pi^1_1$}-conservation of combinatorial principles weaker
              than {R}amsey's theorem for pairs},
   JOURNAL = {Adv. Math.},
  FJOURNAL = {Advances in Mathematics},
    VOLUME = {230},
      YEAR = {2012},
    NUMBER = {3},
     PAGES = {1060--1077},
      ISSN = {0001-8708},
   MRCLASS = {03B30 (03D80 03F30 03F35 05D10)},
  MRNUMBER = {2921172},
MRREVIEWER = {Denis R. Hirschfeldt},
       DOI = {10.1016/j.aim.2012.02.025},
       URL = {https://doi.org/10.1016/j.aim.2012.02.025},
}

@incollection {kolo2021search,
    AUTHOR = {Ko{\l}odziejczyk, Leszek Aleksander and Yokoyama, Keita},
     TITLE = {In search of the first-order part of {R}amsey's theorem for
              pairs},
 BOOKTITLE = {Connecting with computability},
    SERIES = {Lecture Notes in Comput. Sci.},
    VOLUME = {12813},
     PAGES = {297--307},
 PUBLISHER = {Springer, Cham},
      YEAR = {[2021] \copyright 2021},
   MRCLASS = {03B30 (03F35)},
  MRNUMBER = {4347266},
       DOI = {10.1007/978-3-030-80049-9\_27},
       URL = {https://doi.org/10.1007/978-3-030-80049-9_27},
}

@book {dzhafarov2022reverse,
    AUTHOR = {Dzhafarov, Damir D. and Mummert, Carl},
     TITLE = {Reverse mathematics---problems, reductions, and proofs},
    SERIES = {Theory and Applications of Computability},
 PUBLISHER = {Springer, Cham},
      YEAR = {[2022] \copyright 2022},
     PAGES = {xix+488},
      ISBN = {978-3-031-11366-6; 978-3-031-11367-3},
   MRCLASS = {03-02 (03B30 03F35)},
  MRNUMBER = {4472209},
       DOI = {10.1007/978-3-031-11367-3},
       URL = {https://doi.org/10.1007/978-3-031-11367-3},
}

@article {jockusch1972ramsey,
    AUTHOR = {Jockusch, Jr., Carl G.},
     TITLE = {Ramsey's theorem and recursion theory},
   JOURNAL = {J. Symbolic Logic},
  FJOURNAL = {The Journal of Symbolic Logic},
    VOLUME = {37},
      YEAR = {1972},
     PAGES = {268--280},
      ISSN = {0022-4812},
   MRCLASS = {02F35 (02F50)},
  MRNUMBER = {376319},
MRREVIEWER = {John W. Berry},
       DOI = {10.2307/2272972},
       URL = {https://doi.org/10.2307/2272972},
}

@article{liu2012rt22,
  title={RT22 does not imply WKL0},
  author={Liu, Jiayi},
  journal={The Journal of Symbolic Logic},
  volume={77},
  number={2},
  pages={609--620},
  year={2012},
  publisher={Cambridge University Press}
}

@incollection {seetapun1995strength,
    AUTHOR = {Seetapun, David and Slaman, Theodore A.},
     TITLE = {On the strength of {R}amsey's theorem},
      NOTE = {Special Issue: Models of arithmetic},
   JOURNAL = {Notre Dame J. Formal Logic},
  FJOURNAL = {Notre Dame Journal of Formal Logic},
    VOLUME = {36},
      YEAR = {1995},
    NUMBER = {4},
     PAGES = {570--582},
      ISSN = {0029-4527},
   MRCLASS = {03F35 (03C62)},
  MRNUMBER = {1368468},
MRREVIEWER = {Roman Murawski},
       DOI = {10.1305/ndjfl/1040136917},
       URL = {https://doi.org/10.1305/ndjfl/1040136917},
}

@book {hirst1987combinatorics,
    AUTHOR = {Hirst, Jeffry Lynn},
     TITLE = {C{OMBINATORICS} {IN} {SUBSYSTEMS} {OF} {SECOND} {ORDER}
              {ARITHMETIC}},
      NOTE = {Thesis (Ph.D.)--The Pennsylvania State University},
 PUBLISHER = {ProQuest LLC, Ann Arbor, MI},
      YEAR = {1987},
     PAGES = {153},
   MRCLASS = {Thesis},
  MRNUMBER = {2635978},
       URL =
              {http://gateway.proquest.com/openurl?url_ver=Z39.88-2004&rft_val_fmt=info:ofi/fmt:kev:mtx:dissertation&res_dat=xri:pqdiss&rft_dat=xri:pqdiss:8728018},
}

@article {patey2018proof,
    AUTHOR = {Patey, Ludovic and Yokoyama, Keita},
     TITLE = {The proof-theoretic strength of {R}amsey's theorem for pairs
              and two colors},
   JOURNAL = {Adv. Math.},
  FJOURNAL = {Advances in Mathematics},
    VOLUME = {330},
      YEAR = {2018},
     PAGES = {1034--1070},
      ISSN = {0001-8708},
   MRCLASS = {03B30 (03C62 03D80 03F35 03H15 05D10)},
  MRNUMBER = {3787563},
MRREVIEWER = {Fran\c{c}ois G. Dorais},
       DOI = {10.1016/j.aim.2018.03.035},
       URL = {https://doi.org/10.1016/j.aim.2018.03.035},
}

@article {erdos1964representation,
    AUTHOR = {Erd\H{o}s, P. and Moser, L.},
     TITLE = {On the representation of directed graphs as unions of
              orderings},
   JOURNAL = {Magyar Tud. Akad. Mat. Kutat\'{o} Int. K\"{o}zl.},
  FJOURNAL = {A Magyar Tudom\'{a}nyos Akad\'{e}mia. Matematikai Kutat\'{o} Int\'{e}zet\'{e}nek
              K\"{o}zlem\'{e}nyei},
    VOLUME = {9},
      YEAR = {1964},
     PAGES = {125--132},
      ISSN = {0541-9514},
   MRCLASS = {05.60},
  MRNUMBER = {168494},
MRREVIEWER = {J. W. Moon},
}

@article {bovykin2017strength,
    AUTHOR = {Bovykin, Andrey and Weiermann, Andreas},
     TITLE = {The strength of infinitary {R}amseyan principles can be
              accessed by their densities},
   JOURNAL = {Ann. Pure Appl. Logic},
  FJOURNAL = {Annals of Pure and Applied Logic},
    VOLUME = {168},
      YEAR = {2017},
    NUMBER = {9},
     PAGES = {1700--1709},
      ISSN = {0168-0072},
   MRCLASS = {03F25 (03C62 03H15 05D10)},
  MRNUMBER = {3659408},
MRREVIEWER = {Alberto Marcone},
       DOI = {10.1016/j.apal.2017.03.005},
       URL = {https://doi.org/10.1016/j.apal.2017.03.005},
}

@article {lerman2013separating,
    AUTHOR = {Lerman, Manuel and Solomon, Reed and Towsner, Henry},
     TITLE = {Separating principles below {R}amsey's theorem for pairs},
   JOURNAL = {J. Math. Log.},
  FJOURNAL = {Journal of Mathematical Logic},
    VOLUME = {13},
      YEAR = {2013},
    NUMBER = {2},
     PAGES = {1350007, 44},
      ISSN = {0219-0613},
   MRCLASS = {03F35 (03B30)},
  MRNUMBER = {3125903},
MRREVIEWER = {Alberto Marcone},
       DOI = {10.1142/S0219061313500074},
       URL = {https://doi.org/10.1142/S0219061313500074},
}

@article {kreuzer2012primitive,
    AUTHOR = {Kreuzer, Alexander P.},
     TITLE = {Primitive recursion and the chain antichain principle},
   JOURNAL = {Notre Dame J. Form. Log.},
  FJOURNAL = {Notre Dame Journal of Formal Logic},
    VOLUME = {53},
      YEAR = {2012},
    NUMBER = {2},
     PAGES = {245--265},
      ISSN = {0029-4527},
   MRCLASS = {03F35 (03B30)},
  MRNUMBER = {2925280},
MRREVIEWER = {Mariko Yasugi},
       DOI = {10.1215/00294527-1715716},
       URL = {https://doi.org/10.1215/00294527-1715716},
}

@article {monin2021srt22,
    AUTHOR = {Monin, Benoit and Patey, Ludovic},
     TITLE = {{$\mathsf{SRT}^2_2$} does not imply {$\mathsf{RT}^2_2$} in
              {$\omega$}-models},
   JOURNAL = {Adv. Math.},
  FJOURNAL = {Advances in Mathematics},
    VOLUME = {389},
      YEAR = {2021},
     PAGES = {Paper No. 107903, 32},
      ISSN = {0001-8708},
   MRCLASS = {03B30 (03F35 05D10)},
  MRNUMBER = {4288219},
MRREVIEWER = {Carl Mummert},
       DOI = {10.1016/j.aim.2021.107903},
       URL = {https://doi.org/10.1016/j.aim.2021.107903},
}

@incollection {hirschfeldt2008strength,
    AUTHOR = {Hirschfeldt, Denis R. and Jockusch, Jr., Carl G. and
              Kjos-Hanssen, Bj\o rn and Lempp, Steffen and Slaman, Theodore
              A.},
     TITLE = {The strength of some combinatorial principles related to
              {R}amsey's theorem for pairs},
 BOOKTITLE = {Computational prospects of infinity. {P}art {II}. {P}resented
              talks},
    SERIES = {Lect. Notes Ser. Inst. Math. Sci. Natl. Univ. Singap.},
    VOLUME = {15},
     PAGES = {143--161},
 PUBLISHER = {World Sci. Publ., Hackensack, NJ},
      YEAR = {2008},
   MRCLASS = {03D28 (03B30 03F35)},
  MRNUMBER = {2449463},
MRREVIEWER = {A. Ku\v{c}era},
       DOI = {10.1142/9789812796554\_0008},
       URL = {https://doi.org/10.1142/9789812796554_0008},
}

@article {downey2022relationships,
    AUTHOR = {Downey, Rod and Greenberg, Noam and Harrison-Trainor, Matthew
              and Patey, Ludovic and Turetsky, Dan},
     TITLE = {Relationships between computability-theoretic properties of
              problems},
   JOURNAL = {J. Symb. Log.},
  FJOURNAL = {The Journal of Symbolic Logic},
    VOLUME = {87},
      YEAR = {2022},
    NUMBER = {1},
     PAGES = {47--71},
      ISSN = {0022-4812},
   MRCLASS = {03B30 (03D30 03D80)},
  MRNUMBER = {4404619},
MRREVIEWER = {Huishan Wu},
       DOI = {10.1017/jsl.2020.38},
       URL = {https://doi.org/10.1017/jsl.2020.38},
}

@article {jockusch197classes,
    AUTHOR = {Jockusch, Jr., Carl G. and Soare, Robert I.},
     TITLE = {{$\Pi^0_1$} classes and degrees of theories},
   JOURNAL = {Trans. Amer. Math. Soc.},
  FJOURNAL = {Transactions of the American Mathematical Society},
    VOLUME = {173},
      YEAR = {1972},
     PAGES = {33--56},
      ISSN = {0002-9947},
   MRCLASS = {02F30 (02F35 02G05)},
  MRNUMBER = {316227},
MRREVIEWER = {S. Feferman},
       DOI = {10.2307/1996261},
       URL = {https://doi.org/10.2307/1996261},
}

@article{cervelle2000reverse,
  title={The reverse mathematics of CAC for trees},
  author={Cervelle, Julien and Gaudelier, William and Patey, Ludovic},
  journal={The Journal of Symbolic Logic},
  pages={1--23},
  year={2000},
  publisher={Cambridge University Press}
}

@article {montalban2011open,
    AUTHOR = {Montalb\'{a}n, Antonio},
     TITLE = {Open questions in reverse mathematics},
   JOURNAL = {Bull. Symbolic Logic},
  FJOURNAL = {The Bulletin of Symbolic Logic},
    VOLUME = {17},
      YEAR = {2011},
    NUMBER = {3},
     PAGES = {431--454},
      ISSN = {1079-8986},
   MRCLASS = {03B30 (03F35)},
  MRNUMBER = {2856080},
       DOI = {10.2178/bsl/1309952320},
       URL = {https://doi.org/10.2178/bsl/1309952320},
}

@article {patey2017iterative,
    AUTHOR = {Patey, Ludovic},
     TITLE = {Iterative forcing and hyperimmunity in reverse mathematics},
   JOURNAL = {Computability},
  FJOURNAL = {Computability. The Journal of the Association CiE},
    VOLUME = {6},
      YEAR = {2017},
    NUMBER = {3},
     PAGES = {209--221},
      ISSN = {2211-3568},
   MRCLASS = {03B30 (03C62 03D30 03D80 03E40 03F35)},
  MRNUMBER = {3689068},
MRREVIEWER = {Robert S. Lubarsky},
       DOI = {10.3233/COM-160062},
       URL = {https://doi.org/10.3233/COM-160062},
}

@article {liu2022reverse,
    AUTHOR = {Liu, Lu and Patey, Ludovic},
     TITLE = {The reverse mathematics of the thin set and {E}rd{\H{o}}s-{M}oser
              theorems},
   JOURNAL = {J. Symb. Log.},
  FJOURNAL = {The Journal of Symbolic Logic},
    VOLUME = {87},
      YEAR = {2022},
    NUMBER = {1},
     PAGES = {313--346},
      ISSN = {0022-4812},
   MRCLASS = {03D80 (03B30 03F35)},
  MRNUMBER = {4404629},
MRREVIEWER = {Carl Mummert},
       DOI = {10.1017/jsl.2021.98},
       URL = {https://doi.org/10.1017/jsl.2021.98},
}

@misc{houerou2025ramseylike,
Author = {Quentin Le Houérou and Ludovic Patey},
Title = {Ramsey-like theorems for separable permutations},
Year = {2025},
Eprint = {arXiv:2507.07606},
}

@article{patey_ramsey-like_2020,
    title = {Ramsey-like theorems and moduli of computation},
    issn = {0022-4812, 1943-5886},
    url = {https://www.cambridge.org/core/product/identifier/S0022481220000699/type/journal_article},
    doi = {10.1017/jsl.2020.69},
    language = {en},
    urldate = {2020-11-16},
    journal = {The Journal of Symbolic Logic},
    author = {Patey, Ludovic},
    month = oct,
    year = {2020},
    pages = {1--26},
}

@phdthesis{patey2016reverse,
  title={The reverse mathematics of Ramsey-type theorems},
  author={Patey, Ludovic},
  year={2016},
  school={Universit{\'e} Paris Diderot (Paris 7) Sorbonne Paris Cit{\'e}}
}

@article {kjos2009infinite,
    AUTHOR = {Kjos-Hanssen, Bj\o rn},
     TITLE = {Infinite subsets of random sets of integers},
   JOURNAL = {Math. Res. Lett.},
  FJOURNAL = {Mathematical Research Letters},
    VOLUME = {16},
      YEAR = {2009},
    NUMBER = {1},
     PAGES = {103--110},
      ISSN = {1073-2780},
   MRCLASS = {03D32 (03B30 03D80 03F35 60C05)},
  MRNUMBER = {2480564},
MRREVIEWER = {Antonio Montalb\'{a}n},
       DOI = {10.4310/MRL.2009.v16.n1.a10},
       URL = {https://doi.org/10.4310/MRL.2009.v16.n1.a10},
}

@article {greenberg2009lowness,
    AUTHOR = {Greenberg, Noam and Miller, Joseph S.},
     TITLE = {Lowness for {K}urtz randomness},
   JOURNAL = {J. Symbolic Logic},
  FJOURNAL = {The Journal of Symbolic Logic},
    VOLUME = {74},
      YEAR = {2009},
    NUMBER = {2},
     PAGES = {665--678},
      ISSN = {0022-4812},
   MRCLASS = {03D32 (03D80 68Q30)},
  MRNUMBER = {2518817},
MRREVIEWER = {Liang Yu},
       DOI = {10.2178/jsl/1243948333},
       URL = {https://doi.org/10.2178/jsl/1243948333},
}

@article {kjos2011kolmogorov,
    AUTHOR = {Kjos-Hanssen, Bj\o rn and Merkle, Wolfgang and Stephan, Frank},
     TITLE = {Kolmogorov complexity and the recursion theorem},
   JOURNAL = {Trans. Amer. Math. Soc.},
  FJOURNAL = {Transactions of the American Mathematical Society},
    VOLUME = {363},
      YEAR = {2011},
    NUMBER = {10},
     PAGES = {5465--5480},
      ISSN = {0002-9947},
   MRCLASS = {03D28 (03D32 68Q30)},
  MRNUMBER = {2813422},
MRREVIEWER = {Liang Yu},
       DOI = {10.1090/S0002-9947-2011-05306-7},
       URL = {https://doi.org/10.1090/S0002-9947-2011-05306-7},
}

@article {jockusch1989recursively,
    AUTHOR = {Jockusch, Jr., C. G. and Lerman, M. and Soare, R. I. and
              Solovay, R. M.},
     TITLE = {Recursively enumerable sets modulo iterated jumps and
              extensions of {A}rslanov's completeness criterion},
   JOURNAL = {J. Symbolic Logic},
  FJOURNAL = {The Journal of Symbolic Logic},
    VOLUME = {54},
      YEAR = {1989},
    NUMBER = {4},
     PAGES = {1288--1323},
      ISSN = {0022-4812},
   MRCLASS = {03D25},
  MRNUMBER = {1026600},
MRREVIEWER = {Steffen Lempp},
       DOI = {10.2307/2274816},
       URL = {https://doi.org/10.2307/2274816},
}

@book {downey2010algorithmic,
    AUTHOR = {Downey, Rodney G. and Hirschfeldt, Denis R.},
     TITLE = {Algorithmic randomness and complexity},
    SERIES = {Theory and Applications of Computability},
 PUBLISHER = {Springer, New York},
      YEAR = {2010},
     PAGES = {xxviii+855},
      ISBN = {978-0-387-95567-4},
   MRCLASS = {03-02 (03D32 03D80 28A80 60-08 60G46)},
  MRNUMBER = {2732288},
MRREVIEWER = {Bj\o rn Kjos-Hanssen},
       DOI = {10.1007/978-0-387-68441-3},
       URL = {https://doi.org/10.1007/978-0-387-68441-3},
}

@article{khan_forcing_2017,
    title = {Forcing with {Bushy} {Trees}},
    url = {http://arxiv.org/abs/1503.08870},
    language = {en},
    urldate = {2020-11-16},
    journal = {arXiv:1503.08870 [math]},
    author = {Khan, Mushfeq and Miller, Joseph S.},
    month = mar,
    year = {2017},
    note = {arXiv: 1503.08870},
}

@article{pateylowness,
  title={Lowness and avoidance},
  author={Patey, Ludovic},
  journal={preparation. url: https://ludovicpatey.com/lowness-avoidance}
}

@article{simpson_partial_1988,
    title = {Partial realizations of {Hilbert}'s {Program}},
    volume = {53},
    issn = {0022-4812},
    url = {http://dx.doi.org/10.2307/2274508},
    doi = {10.2307/2274508},
    number = {2},
    journal = {The Journal of Symbolic Logic},
    author = {Simpson, Stephen G.},
    year = {1988},
    mrnumber = {947843 (89h:03005)},
    pages = {349--363},
}

@incollection{yokoyama_pi_11_2010,
    title = {On {\textbackslash}{Pi}\_1{\textasciicircum}1 conservativity for {\textbackslash}{Pi}\_2{\textasciicircum}1 theories in second order arithmetic},
    url = {https://doi.org/10.1142/9789814293020_0016},
    booktitle = {10th {Asian} {Logic} {Conference}},
    publisher = {World Sci. Publ., Hackensack, NJ},
    author = {Yokoyama, Keita},
    year = {2010},
    mrnumber = {2798907},
    doi = {10.1142/9789814293020_0016},
    pages = {375--386},
}

\end{document}